\documentclass[a4paper, 11pt]{article}
\usepackage[top=2.5cm, bottom=2.5cm, left=2.25cm, right=2.25cm]{geometry}
\usepackage[UKenglish]{babel}

\usepackage[usenames, dvipsnames]{xcolor}

\usepackage[utf8]{inputenc}
\usepackage[T1]{fontenc}
\usepackage{amsfonts}
\usepackage{enumitem,csquotes}
\usepackage{amsmath}
\usepackage{hhline}
\usepackage{amsthm}
\usepackage{amssymb}
\usepackage{array,multirow,makecell}

\usepackage{fancybox}

\newcolumntype{L}[1]{>{\raggedright\let\newline\\\arraybackslash\hspace{0pt}}m{#1}}
\newcolumntype{C}[1]{>{\centering\let\newline\\\arraybackslash\hspace{0pt}}m{#1}}
\newcolumntype{R}[1]{>{\raggedleft\let\newline\\\arraybackslash\hspace{0pt}}m{#1}}
\newcolumntype{P}[1]{>{\centering\arraybackslash}p{#1}}

\usepackage{framed}
\usepackage{caption}
\usepackage{bbm}
\usepackage{graphicx}
\usepackage{subcaption}
\usepackage{url}
\usepackage{here}
\usepackage{varioref}
\usepackage{enumitem} 

\usepackage{tikz}
\usetikzlibrary{automata,intersections,shapes,arrows,calc,positioning,decorations}
\usepackage{pgfplots}
\usepackage{tkz-fct}

\newtheorem{theorem}{Theorem}[section]
\newtheorem{proposition}[theorem]{Proposition}
\newtheorem{corollary}[theorem]{Corollary}
\newtheorem{lemma}[theorem]{Lemma}
\theoremstyle{definition}

\theoremstyle{remark}
\newtheorem{remark}[theorem]{Remark}
\newtheorem{example}[theorem]{Example}

\usepackage{mathtools} 
\usepackage{extarrows}

\newcommand{\mB}{\mathbb B}

\newcommand{\mR}{\mathbb R}
\newcommand{\mZ}{\mathbb Z}
\newcommand{\mN}{\mathbb N}

\newcommand{\mL}{\mathbb L}

\newcommand{\sB}{{\mathcal B}}
\newcommand{\sC}{{\mathcal C}}

\newcommand{\sL}{{\mathcal L}}
\newcommand{\sK}{{\mathcal K}}

\newcommand{\sP}{{\mathcal P}}

\newcommand{\bG}{{\bf G}}

\newcommand{\bpm}[1]{\begin{pmatrix}#1\end{pmatrix}}

\newcommand{\gs}{\gtrsim}

\newcommand{\rd }{{\rm  d}}

\newcommand{\e}{{\rm  e}}
\newcommand{\eps}{\varepsilon}

\newcommand{\fpd}[2]{#1 \lvert #1 \rvert^{#2}}

\numberwithin{equation}{section}
\numberwithin{figure}{section}
\numberwithin{table}{section}

\usepackage[urlcolor = black, colorlinks = true, citecolor=black, linkcolor = black]{hyperref}
\usepackage{tikz}
\usetikzlibrary{automata,intersections,shapes,arrows,calc,positioning,decorations,backgrounds}
\usepackage{pgfplots}
\pgfplotsset{compat=newest}
\usepackage{tkz-fct}
\usetikzlibrary{matrix}
\usepgfplotslibrary{fillbetween}
\tikzstyle{sum} = [draw, fill=white, thick, circle, inner sep=2pt,minimum size=2pt]
\tikzstyle{sum_black} = [draw, fill=black, thick, circle, inner sep=1pt,minimum size=1pt]

\title{Passivity theorems for input-to-state stability of forced {L}ur'e inclusions and equations, and consequent entrainment-type properties}

\author{Chris Guiver\thanks{School of Computing, Engineering and the Built Environment, Edinburgh Napier University, UK. Email: \texttt{c.guiver@napier.ac.uk}}}
\date{June 2024}	

\makeatletter
\let\thetitle\@title
\let\theauthor\@author
\let\thedate\@date
\makeatother

\begin{document}
\maketitle

\begin{abstract}
A suite of input-to-state stability results are presented for a class of forced differential inclusions, so-called Lur'e inclusions. As a consequence, semi-global incremental input-to-state stability results for systems of forced Lur'e differential equations are derived. The results are in the spirit of the passivity theorem from control theory as both the linear and nonlinear components of the Lur'e inclusion (or equation) are assumed to satisfy passivity-type conditions. These results provide a basis for the analysis of forced Lur'e differential equations subject to (almost) periodic forcing terms and, roughly speaking, ensure the existence and attractivity of (almost) periodic state- and output-responses, comprising another focus of the present work. One ultimate aim of the study is to provide a robust and rigorous theoretical foundation for a well-defined and tractable ``frequency response'' of forced Lur'e systems.
\end{abstract}



\section{Introduction}

%
%
We consider stability properties of the class of feedback connections depicted in Figure~\ref{fig:lure}, comprising a linear system $\Sigma$ in the forward path and a static $F$ nonlinearity in the feedback path. Such configurations are often called forced Lur'e (also Lurie or Lurye) systems, and the nonlinearity $F$ may be a correspondence (a set-valued map) or a usual function (a singleton-valued map), leading to Lur'e inclusions or Lur'e equations, respectively. The term $v$ in Figure~\ref{fig:lure} is an exogenous signal which we call a forcing term and, depending on the context, may denote a control or a disturbance.

%
%
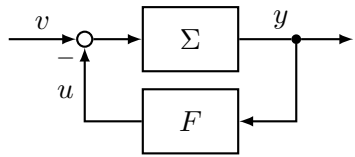
\begin{figure}[h!]
\centering
	\begin{tikzpicture}
		\coordinate (O) at (0,0);
		\node[draw, thick, minimum width=0.5cm, minimum height=0.85cm, anchor=south west, text width=1cm, align=center] (P) at (O) {$\Sigma$};
		\node[draw, thick, minimum width=0.5cm, minimum height=0.85cm, text width=1cm, align=center] (C) at ($(P.270) - (0,0.65)$) {$F$};
		\node[sum_black] (Y) at ($(P.0) +(0.75,0)$) {};
  	\node[sum] (U) at ($(P.180) - (0.75,0)$) {};
   \node[] (Up) at ($(U) - (0.25,0.25)$) {{\footnotesize$-$}};
		\draw[thick] (P) -- (Y.180)node[above, pos=0.8] {$y$};
		\draw[thick,-latex] (Y) -- ($(Y)+(0.75,0)$);
		\draw[thick,-latex] (Y.270) |- (C.0); 
	\draw[thick, -latex] ($(U)-(1,0)$) -- (U)node[above, pos=0.5] {$v$};	
  %
		%
		\draw[thick,-latex] (C.180)  -|  (U) node [left, pos=0.7]{$u$};
		\draw[thick, -latex] (U) -- (P.180);
	\end{tikzpicture}
	\caption{Forced Lur'e system}
	\label{fig:lure}
\end{figure}
%

%
%
 Lur'e systems date back to the 1940s, and the work of the Soviet scientists A.\ I.\ Lur'e and V.\ N.\ Postnikov~\cite{MR0011360}. Their study has, in part, been popularized by the famous Aizerman Conjecture~\cite{aizerman1949}, which is false in general with counterexamples appearing in~\cite{fitts1966two}, 
yet is a conjecture with a long history in automatic control. It, along with the subsequent Kalman Conjecture~\cite{kalman1957physical}, has been considered in both the former Soviet and the Western literature --- see, for example,~\cite{boiko2020counter,
leonov2010algorithm,
seiler2020construction} for a brief overview.
%
%
More generally, the study of stability properties, meaning a number of different possible notions, of Lur'e systems is called {\em absolute stability theory}, and usually adopts the position that the nonlinear term $F$ is uncertain. Consequently, sufficient conditions for stability are sought which are valid for a class of nonlinear terms, usually via the the interplay of frequency-domain properties of the linear component and sector properties of the class of nonlinearities. Results in absolute stability theory traditionally come in one of two strands. The first adopts Lyapunov approaches to deduce global asymptotic stability of unforced Lur'e systems (that is, $v=0$), as is the approach across~\cite{MR2381711,MR2116013,k02}. The second strand is an input-output approach, pioneered by Sandberg and Zames in the 1960s, to infer $L^2$- and $L^{\infty}$-stability; see, for example,~\cite{MR0490289,MR1946479}. 

Absolute stability theory has garnered much academic interest (see, for instance, the extensive literature reviewed in~\cite{liberzon2006essays}). Indeed, the literature related to the stability and convergence properties of Lur'e systems is expansive and, writing in 2024, is still an active area of research.  Of a number of absolute stability criteria, we make specific mention of the classical passivity theorems, see for example~\cite[Chapters 6,7]{k02} (more specifically, in the ``infinite sector case'' such as~\cite[Theorem 7.1, p.\ 265]{k02} --- there with~$K_1 = 0$). In this setting, the passivity theorem is also called the positivity theorem by some authors, as in~\cite[Theorem 4.1]{haddad1993explicit}. In any case, these results play an inspirational role in the present development.  They are essentially nonlinear generalizations of the passivity theorem for linear control systems which traces its roots back to the work of Zames~\cite{zames1966input} and, recall, states that the feedback connection of an $L^2$-input-output stable and strictly passive positive real plant, and passive controller, is itself passive and $L^2$-input-output stable. 

%
%
Input-to-State Stability (ISS), dating back to the 1989 work of Sontag~\cite{sontag1989smooth}, arose a suite of stability concepts for controlled (or forced) systems of nonlinear differential equations. In overview, ISS ensures natural boundedness properties of the state, in terms of (functions of) the initial conditions and forcing terms. Practically, ISS ensures that ``small'' external noise, disturbance or unmodelled dynamics result in a correspondingly ``small'' effect on the resulting state. ISS has been well-studied since its inception with subsequent developments in the 1990s by Sontag and others across, for example,~\cite{jiang1994small,sontag1995characterizations,sontag1996new}, and a number of variants appearing including so-called {\em integral} ISS~\cite{angeli2000characterization,sontag1998comments}, {\em strong} ISS~\cite{chaillet2014combining}, {\em exponential} ISS~\cite{MR4591532} and {\em finite-time} ISS~\cite{MR4688194,MR2665472}. The ISS property has been developed in discrete time, from~\cite{jiang2001input} onwards,  as well as for switched systems~\cite{xie2001input}, and infinite-dimensional control systems with recent reviews~\cite{MR4131339,mironchenko2024input}. There is now a vast literature on ISS, with surveys including~\cite{dashkovskiy2011input,sontag2008input}, as well as the recent monograph~\cite{MR4592570}. One strength of the input-to-state stability paradigm is that it both encompasses and unifies asymptotic and input-output approaches to stability, indeed, to quote~\cite[Preface]{karafyllis2011stability}: {\em ``ISS has bridged the gap which previously existed between the input–output and the state-space methods, two popular approaches within the control systems community.''} 
%
%
As documented in~\cite{MR4720562}, considerable effort has been dedicated to establishing ISS properties for Lur'e systems, originating in the seminal work~\cite{arcak2002input}. To summarise, many classical absolute stability criteria, when suitably adapted, do ensure ISS and ISS-type properties; see, for example~\cite{MR4720562,MR4023134,jayawardhana2009input,jayawardhana2011circle}.

%
%
Incremental stability is broadly concerned with bounding the difference of two arbitrary trajectories of a given system in terms of the difference between initial states and, if included, input terms; see, for instance~\cite{angeli2002lyapunov,angeli2009further,sepulchre2022incremental}. These aims align with those of contraction analysis or contraction theory as in~\cite{aminzare2014contraction,
MR1632102,MR4348709} or the recent text~\cite{FB-CTDS}, and also of convergent systems; see~\cite{rwm13} and the references therein. The works~\cite{MIC-2010-3-2,rwm13} compare and contrast incremental stability, contraction analysis, and convergent systems concepts. For linear control systems, incremental stability concepts coincide with their corresponding stability versions via the superposition principle, and this is not the case in general for nonlinear systems. The importance of incremental stability properties is recognised across works such as~\cite{chaffey2023monotone,forni2018differential,ofir2024contraction,sepulchre2022incremental}. By way of utility, incremental stability results provide a toolbox for nonlinear observer design, the study of synchronization-type problems, and the analysis of convergent or (almost) periodic inputs in forced differential equations~\cite{gilmore2021incremental}, a property broadly termed called ``entrainment'' (see, for instance~\cite[Chapter 7]{jordan2007nonlinear}). Recent applications of incremental stability include to establishing robustness properties of certain neural network architectures~\cite{centorrino2023euclidean,revay2020lipschitz}, as well as those discussed in~\cite{giaccagli2023further}. 

%
%
Here, we establish a suite of ISS results for Lur'e differential inclusions in Section~\ref{sec:inclusion}. These, and all our results, are in the spirit of the above-mentioned passivity/positivity theorem for absolute stability, and since extended to ensure ISS for Lur'e differential equations in~\cite[Theorem 1]{arcak2002input}. Results from~\cite{teel2000smooth} play an essential role in our argumentation for Lur'e inclusions.
In addition to their own interest, we consider Lur'e inclusions to facilitate our second set of main results, namely, semi-global incremental (integral) ISS results for Lur'e differential equations, the subject of Section~\ref{sec:incremental}. Technically, the treatment of Lur'e inclusions provides a convenient framework for addressing certain uniformity issues which arise when establishing incremental stability properties of whole families of Lur'e differential equations.  The semi-global incremental (integral) ISS concept here means that, for arbitrary bounded sets of initial conditions and forcing terms there exist comparison functions such that an incremental (integral) ISS estimate holds, and is the same concept which appears in~\cite{MR4052324,gilmore2021incremental}. For all practical purposes, semi-global stability notions seem to be at least as interesting as global stability concepts, and impose much less restrictive conditions on the model data, as we discuss.  As already mentioned, semi-global incremental (integral) ISS results in turn provide a framework for rigorously studying the (almost or asymptotically almost) periodic state- and output-response of forced Lur'e equations to (almost or asymptotically almost) periodic forcing terms, in the sense of Bohr or, more generally, in the sense of Stepanov (see, for example \cite[Section 4.7]{MR0275061}). Our convergence results in this setting appear in Section~\ref{sec:ap}. 

%
%
Thus, on the one hand, our work continues the analysis of almost periodic differential equations from the perspective of mathematical systems and control theory. On the other hand, the current work lays the theoretical foundations for a well-defined and tractable ``frequency response'' of forced Lur'e systems, beyond that of the harmonic balance method~\cite{gilmore1991nonlinear}. Indeed, one motivating application is to designing extracted-energy maximizing controls for wave-energy converters (WECs), a timely renewable electrical energy generation problem, in a field where frequency-domain methods are  prevalent; see the comments in~\cite{bacelli2020comments}, for instance. Of course, this traditionally requires linear control systems, so that the ability to move beyond linear systems and treat nonlinear models for WECs is important, as argued and evidenced in~\cite{windt2021reactive}. Very briefly, heaving point-absorber WECs may be expressed as forced Lur'e differential equations, with forcing terms which are likely to be almost periodic corresponding to the wave excitation force~\cite{merigaud2017free}, whence fall into the scope of the present work. However, in order to maintain a singular focus and reasonable length, the application of the present results to frequency domain methods in nonlinear wave-energy conversion problems appears in the conference proceedings~\cite{guiver2024ccta}.

%
%
In terms of scope of results, there is overlap with the present work and others including the author, namely~\cite{gilmore2020infinite,MR4052324,gilmore2021incremental}. The key novelty and contribution of the present work is technical, yet is important, and is that the nonlinearities considered here may belong to an infinite sector, such as ``superlinear'' nonlinearities $F(z) = \pm z \lvert z \rvert^d$ (sign depending on feedback convention used) which find relevance in a wave-energy applications perspective~\cite{fusco2014robust}. The works~\cite{gilmore2020infinite,MR4052324,gilmore2021incremental} all impose linear boundedness of the nonlinearity (in other words, a finite-sector condition is imposed, and this structure is essential).

%
%
The content of Sections~\ref{sec:inclusion}, \ref{sec:incremental} and~\ref{sec:ap} has been outlined above. The rest of this work is organised as follows. Preliminary material appears in Section~\ref{sec:preliminaries}. Examples are presented in Section~\ref{sec:examples}, and summarising comments appear in Section~\ref{sec:summary}. The proofs of certain technical results are relegated to the Appendix.


\subsection*{Notation}

We use standard mathematical notation, and highlight only a few matters. We let~$\| \cdot \|$ denote the usual two-norm on Euclidean space~$\mR^n$, induced by the usual inner product on~$\mR^n$ denoted~$\langle \cdot, \cdot \rangle$. A square matrix is called Hurwitz (also sometimes asymptotically stable) if every eigenvalue has negative real part. The notation $\mB(z,r)$ denotes the open ball in Euclidean space (consistent with the context) with center $z$ and radius $r$.

As usual, we let~$\sP$ denote the set of functions~$\mR_+ \to \mR_+$ which are continuous, zero at zero, and positive for positive arguments (such functions are also sometimes called positive definite). We let~$\sK$ denote the subset of~$\sP$ of strictly increasing functions. The subset of~$\sK$ of unbounded functions is denoted~$\sK_\infty$ (such functions are also sometimes called proper).  We denote by
$\sK\sL$  the set of all functions~$\psi: \mR_+\times \mR_+\to\mR_+$ with the following properties: for each fixed~$t\in\mR_+$, the function~$\psi(\,\cdot\,,t)$ is in~$\sK$, and for each fixed~$s\in\mR_+$, the function~$\psi(s,\,\cdot\,)$ is non-increasing and~$\phi(s,t)\to 0$ as~$t\to\infty$.

Since we consider upper bounds in terms of~$\sK_\infty$ functions, by~\cite[Lemma 2.5]{MR1643670}, we may assume that such upper bounds are continuously differentiable (in fact they may be chosen to be infinitely differentiable). Similarly, we shall make use of lower bounds in terms of functions in~$\sP$. A combination of~\cite[Lemma 18]{MR3245919} and~\cite[Lemma 2.5]{MR1643670} means that we may assume that such lower bounds are also continuously differentiable.

We call a function~$g : \mR_+ \to \mR_+$ {\em linearly bounded} if there exists~$a >0$ such that~$\lvert g(s) \rvert \leq a \lvert s \rvert$ for all~$s \in \mR_+$.

\section{Lur'e differential inclusions --- preliminaries }\label{sec:preliminaries}

We gather preliminary material on Lur'e differential inclusions.

Set~$\mL:=\mR^{n \times n} \times \mR^{n \times m} \times \mR^{m\times n}$ for fixed positive integers~$m$ and~$n$.
Consider the forced Lur'e differential inclusion
\begin{equation}\label{eq:lure_inclusion}
\dot x(t)-Ax(t)-Bv(t)\in -BF(Cx(t)),\quad t\geq 0,
\end{equation}
where~$(A,B,C)\in\mL$,~$v\in L^1_{\rm loc}(\mR_+,\mR^m)$, and~$F: \mR^m \rightrightarrows \mR^m$ is a correspondence (also termed a set-valued map). Except where specifically mentioned, we shall assume that~$F$ is upper hemicontinuous (also called upper semicontinuous by some authors, see~\cite[p.\ 558]{MR2378491}) which has values that are non-empty, compact and convex subsets of~$\mR^m$. We say that~\eqref{eq:lure_inclusion} is {\em unforced} if~$v=0$. Occasionally, we will refer to~\eqref{eq:lure_inclusion} as Lur'e inclusion~$(A,B,C,F)$.

Let~$v\in L^1_{\rm loc}(\mR_+,\mR^m)$ be given and let~$0<\tau\leq\infty$. A function~$x\in W^{1,1}_{\rm loc}([0,\tau),\mR^n)$ is said to be a solution of~\eqref{eq:lure_inclusion} on~$[0,\tau)$ if~\eqref{eq:lure_inclusion} holds for almost every~$t\in[0,\tau)$. A solution of~\eqref{eq:lure_inclusion} on~$\mR_+ := [0,\infty)$ is called a global solution. We define two sets of behaviours of~\eqref{eq:lure_inclusion} by
\begin{align*}
  \sB_{\rm inc}&:= \big\{(v,x)\in L^1_{\rm loc}(\mR_+,\mR^m)\times
    W^{1,1}_{\rm loc}(\mR_+,\mR^n)\; : \; \text{$(v,x)$ satisfies~\eqref{eq:lure_inclusion}
a.e. on~$\mR_+$}\big\} \\
\text{and} \quad \sB_{\rm inc}^\infty &:=\big\{ (v,x)\in  \sB_{\rm inc}  \; : \; v\in L^\infty_{\rm loc}(\mR_+,\mR^m) \big\}\,.
\end{align*}
An element~$(v,x)\in\sB_{\rm inc}$ is also called a {\em trajectory} of~\eqref{eq:lure_inclusion}. We shall occasionally write~$\sB_{\rm inc}(F)$ or $\sB_{\rm inc}^\infty(F)$ to emphasise the choice of inclusion~$F$ in~\eqref{eq:lure_inclusion}. Note that~$(0,0) \in \sB_{\rm inc}$ when~$F(0) = \{0\}$.

Regarding the existence of solutions of~\eqref{eq:lure_inclusion}, we record the following fact, which is a consequence of~\cite[Corollary 5.2]{MR1189795}:
\begin{enumerate}[label = {\bfseries (\Alph*)}, start = 6]
    \item \label{ls:F} for all~$v\in L^1_{\rm loc}(\mR_+,\mR^n)$ and every~$x_0\in\mR^n$, there exist~$0<\tau\leq\infty$ and a solution~$x$ of~\eqref{eq:lure_inclusion} on~$[0,\tau)$ such that~$x(0)=x_0$. If~$\tau<\infty$ and~$x$ is bounded, then the solution~$x$ can be extended beyond~$\tau$, that is,    there exist~$\tau<\tilde\tau\leq\infty$ and a solution~$\tilde x$ of~\eqref{eq:lure_inclusion} on~$[0,\tilde\tau)$ such that~$\tilde x(t)= x(t)$ for all~$t\in[0,\tau)$.
\end{enumerate}
We recall relevant Input-to-State Stability concepts. Assuming that~$F(0) = \{0\}$, the Lur'e inclusion~\eqref{eq:lure_inclusion} is called:
\begin{itemize}
    \item Integral Input-to-State Stable (IISS) with linear IISS gain if there exist~$\psi\in\sK\sL$ and~$\phi\in\sK$ such that
    \begin{equation}\label{eq:iiss_inclusion}
\|x(t)\|\leq\psi(\|x(0)\|,t)+\phi\Big(\int_0^t\|v(s)\|{\rm d}s\Big)\quad\forall\:t\geq 0,\,\,\forall\:
(v,x)\in\sB_{\rm inc}\,.
\end{equation}
\item Input-to-State Stable (ISS) if there exist~$\psi\in\sK\sL$ and~$\phi\in\sK$ such that
\begin{equation}\label{eq:iss_inclusion}
\|x(t)\|\leq \psi(\|x(0)\|,t)+\phi(\|v\|_{L^\infty(0,t)})\quad\forall\:t\geq 0,\,\,\forall\:
(v,x)\in\sB_{\rm inc}^\infty\,. 
\end{equation}
\item Exponentially Input-to-State Stable if there exist~$M, \gamma >0$ such that
\begin{equation}\label{eq:exp_iss_inclusion}
\|x(t)\|\leq M\big(\e^{-\gamma t} \|x(0)\| + \|v\|_{L^\infty(0,t)}\big)\quad\forall\:t\geq 0,\,\,\forall\:
(v,x)\in\sB_{\rm inc}^\infty\,. 
\end{equation}
\end{itemize}
In other words, the exponential ISS property is an ISS estimate of the form~\eqref{eq:iss_inclusion} with
\begin{equation}\label{eq:exp_iss_comparison}
    \phi(s,t) := M\e^{-\mu t} s \quad \text{and} \quad \psi(s) := M s \quad \forall \: (s,t) \in \mR_+\times \mR_+\,.
\end{equation}
Observe that the above definitions involve bounds which are required to hold for trajectories of~\eqref{eq:lure} --- these are vacuously satisfied if the behaviours of~\eqref{eq:lure_inclusion} are empty. Since stability fundamentally relates to long-term properties of solutions, it is therefore arguably more accurate to include the requirement in the above definitions that all solutions are global, at the cost of clunkier formulations. 

%
%
The following result demonstrates that existence of suitable so-called ISS Lyapunov functions is sufficient for extendability of solutions to global solutions, and various ISS notions. Since constructing Lyapunov functions comprises our present line of argument, the issues discussed above are moot.

\begin{lemma}\label{lem:inclusion_ISS_lyapunov}
Consider the Lur'e inclusion~\eqref{eq:lure_inclusion} and introduce the statement:
Assume that there exist continuously differentiable~$V : \mR^n \to \mR_+$ and~$\alpha_i \in \sK_\infty$,~$i= 1,2$, such that
\[ \alpha_1(\| z \|) \leq V(\|z\|) \leq \alpha_2(\| z\|) \quad \forall \: z \in \mR^n\,.\]
The following statements hold.
\begin{enumerate}[label = {\rm \bfseries (\arabic*)}, ref = {\rm \bfseries (\arabic*)}]
\item \label{ls:ISS_lyapunov} If there exist~$\alpha_i \in \sK_\infty$,~$i= 3,4$ such that
\begin{equation}\label{eq:iss_lie_dervative}
     \langle (\nabla V)(z), Az - B(w - u)\rangle  \leq -\alpha_3(\|z\|) + \alpha_4(\|u \|) \quad \forall \: w \in F(Cz), \; \forall \: (u,z) \in \mR^m \times \mR^n\,, 
     \end{equation}
then, for every~$v \in L^\infty_{\rm loc}(\mR_+,\mR^n)$, every solution of~\eqref{eq:lure_inclusion} may be extended to a global solution, and~\eqref{eq:lure_inclusion} is ISS. 
\item \label{ls:IISS_lyapunov}  If there exist are~$\mu\in\sP$ and~$\beta>0$ such that
\begin{equation}\label{eq:iiss_lie_dervative}
\langle(\nabla V)(z),Az - B(w - u) \rangle\leq-\mu(\|z\|)+\beta\|u\| \quad \forall \: w \in F(Cz), \; \forall \: (u,z) \in \mR^m \times \mR^n\,,
\end{equation}
then, for every~$v \in L^1_{\rm loc}(\mR_+,\mR^n)$, every solution of~\eqref{eq:lure_inclusion} may be extended to a global solution, and~\eqref{eq:lure_inclusion} is ISS, and~\eqref{eq:lure_inclusion} is IISS with linear IISS gain.
\item If~\ref{ls:ISS_lyapunov} holds and every~$\alpha_i$ may be chosen to be quadratic, that is,~$\alpha_i(s) = c_i s^2$ for~$i = 1,2,3,4$, for positive constants~$c_i >0$, then~\eqref{eq:lure_inclusion} is exponentially ISS. 
\end{enumerate}
\end{lemma}
%
%
\begin{proof}
Let $v \in L^1_{\rm loc}(\mR_+,\mR^m)$ and fix $x(0) \in \mR^n$. Let $x$ denote the corresponding unique solution of~\eqref{eq:lure_inclusion}, which is defined on $[0,\tau)$ for some $0 < \tau \leq \infty$, the existence of which follows from fact~\ref{ls:F}.

Arguing as in the case for differential equations, the existence of $V$ as in statements~\ref{ls:ISS_lyapunov} and~\ref{ls:IISS_lyapunov} ensures that $x$ is bounded on its interval of definition by some $\rho >0$ which is independent of this interval. In the case of statement~\ref{ls:ISS_lyapunov}, we additionally require~$v \in L^\infty_{\rm loc}(\mR_+,\mR^m)$, as imposed.  It now follows from fact~\ref{ls:F} again that $x$ may be extended to a global solution.

The argument that~\eqref{eq:lure_inclusion} is ISS or IISS now proceeds analogously as in the case of differential equations; see, for example, the proof of~\cite[Theorem 5.41]{MR3288478}) or, in the IISS case,~\cite[Theorem 1]{angeli2000characterization}. 
\end{proof}

\section{A passivity theorem for Lur'e differential inclusions}\label{sec:inclusion}

Here we present sufficient conditions for a class of Lur'e inclusions to admit a range of input-to-state stability properties. Our first main result is Theorem~\ref{thm:iss_inclusion} below. Before stating the theorem, we describe the correspondences which define the Lur'e differential inclusions presently under consideration. Fix continuously differentiable~$\theta \in \sK_\infty$,~$\alpha : \mR_+ \to \mR_+$ with~$\alpha \leq \theta$,~$c, \mu >0$,~$c \mu \geq 1$, and define~$F : \mR^m \rightrightarrows \mR^m$ by
\begin{equation}\label{eq:F_passive_inclusion}
    F(y) : = \big\{ w \in \mR^m \: : \: \| w \| \leq \theta(\|y\|),  \; \langle w,y\rangle \geq \| y \| \alpha(\|y \|), \; \& \; c\langle w,y\rangle \geq \| w\| \; \text{if~$\| y \| > \mu$}\big\}\,.
\end{equation}
It shall be convenient to record the defining properties of~$F(y)$ individually, that is, 
\begin{subequations}\label{eq:F_sector}
\begin{align}
\| w \| &\leq \theta(\|y\|) & \forall \: w &\in F(y), \; \forall \: y \in \mR^m\,, \label{eq:F_sector_0} \\
    \langle w,y\rangle &\geq \| y \| \alpha(\|y \|) &\forall \: w &\in F(y), \; \forall \: y \in \mR^m \,, \label{eq:F_sector_1} \\
    \text{and} \quad c\langle w,y\rangle &\geq \| w\| &  \forall \: w &\in F(y), \; \forall \: y \in \mR^m \; \text{with~$\| y\| > \mu$}\,. \label{eq:F_sector_2}
\end{align}
\end{subequations}
Some remarks are in order. The hypothesis~$c \mu \geq 1$ is necessary~\eqref{eq:F_sector_2} to hold.
The assumption that~$\alpha(s) \leq \theta(s)$ implies that~$F(y)$ is non-empty for all~$y \in \mR^m$. Further, condition~\eqref{eq:F_sector_0} entails that~$F(0) = \{0\}$ and also that~$F(y)$ is bounded for every~$y \in \mR^m$. Since~$F(y)$ is closed for every~$y \in \mR^m$, we conclude that~$F$ takes compact values.  Furthermore, the correspondence~$F$ is convex-valued and upper hemicontinuous. Proofs of these claims appear in the Appendix.

The properties~\eqref{eq:F_sector} defining~$F$ are set-valued versions of the hypotheses appearing on the nonlinear function~$\phi$ in~\cite[Theorem 1]{arcak2002input}. It is clear that hypothesis~\eqref{eq:F_sector_2} follows immediately from~\eqref{eq:F_sector_1} with~$\mu = c = 1$ in the case~$m=1$. Indeed, here
\[ w y = \langle w, y \rangle \geq | y | \alpha(|y |) \geq 0 \quad \forall \: w \in F(y), \; \forall \: y \in \mR\,,\]
and so~$\lvert w \rvert \leq \lvert w y \rvert  =w y = \langle w, y \rangle$ whenever~$\lvert y \rvert >1$. This redundancy of~\eqref{eq:F_sector_2} in the~$m=1$ case parallels the situation in~\cite{arcak2002input}, as noted in~\cite[Remark 2]{arcak2002input}.

%
%
\begin{theorem}\label{thm:iss_inclusion}
Consider the Lur'e inclusion~\eqref{eq:lure_inclusion} with~$F$ as in~\eqref{eq:F_passive_inclusion}.  Assume that the pair~$(C,A)$ is detectable, and that there exists a symmetric positive semi-definite~$P \in \mR^{n\times n}$ such that
\begin{equation}\label{eq:P_lmi_condition}
\bpm{A^\top P + PA & PB - C^\top \\ B^\top P - C & 0 } \leq 0\,.
\end{equation}
The following statement holds:
\begin{enumerate}[label = {\rm \bfseries (\arabic*)}, ref = {\rm \bfseries (\arabic*)}]
\item \label{ls:inclusion_ISS} if~$\alpha \in \sK_\infty$, then~\eqref{eq:lure_inclusion} is ISS.
\end{enumerate}
Now assume that~$\alpha \in \sK_\infty$ admits the lower bound~$\alpha(s) \geq \eps_0 s$ for some~$\eps_0 >0$. Furthermore,
\begin{enumerate}[resume*]
\item \label{ls:inclusion_sg_exp_ISS}  for every~$R>0$, there exist~$M, \gamma >0$ such that inequality~\eqref{eq:exp_iss_inclusion} holds for all~$x(0) \in \mR^n$ and~$v\in L^\infty(\mR_+,\mR^m)$ with~$\|x(0)\| + \| v\|_{L^\infty} \leq R$.
\item \label{ls:inclusion_exp_ISS} if~$\theta$ is linearly bounded, then~\eqref{eq:lure_inclusion} is exponentially ISS. 
\end{enumerate}
 The ISS comparison functions and constants in~\eqref{eq:iiss_inclusion}--\eqref{eq:exp_iss_inclusion} arising in the above conclusions depend only on the linear data~$(A,B,C)$, the bounding terms~$\alpha, \theta, \mu, c$ and, where relevant, the fixed positive constant~$R$.
\end{theorem}
%
%
Statement~\ref{ls:inclusion_ISS} is a generalisation of~\cite[Theorem 1]{arcak2002input} to a differential inclusion setting, and from which the present result draws inspiration. The conclusion in statement~\ref{ls:inclusion_sg_exp_ISS} is the so-called semi-global exponential ISS property. Whilst the previous differential inclusion result is of independent interest, our primary motivation for Theorem~\ref{thm:iss_inclusion} is to facilitate incremental ISS results for families of forced systems Lur'e differential equations. The uniformity of comparison functions/constants ensured by the theorem is consequently essential. 

%
%
We comment on the condition~\eqref{eq:P_lmi_condition}. Let 
\[ \bG(s) := C(sI-A)^{-1}B\,, \] 
a proper rational function of the complex variable $s$, denote the {\em transfer function} associated with the Lur'e inclusion~\eqref{eq:lure}. If~$(A,B,C)$ is controllable and observable, and~$\bG$ is positive real (see, for instance~\cite{guiver2017transfer}), then a positive definite solution~$P = P^\top$ to the LMI~\eqref{eq:P_lmi_condition} follows from the standard Positive Real Lemma, and can be found in, for example~\cite[Corollary 5.6]{MR2381711} or~\cite[Theorem 13.26]{zhou1996robust}. The conditions that~$A$ is Hurwitz and~$\bG$ are positive real are sufficient for the existence of a symmetric positive semi-definite solution~$P = P^\top$ to~\eqref{eq:P_lmi_condition}. Whilst we expect that this result is available in the literature, we could not find a suitable statement in the desired form. It may be derived by applying the Bounded Real Lemma appearing as~\cite[Theorem 3.7.1, note, p.\ 111]{Green:1994:LRC:191301} to the Cayley transform $(I-\bG)(I+\bG)^{-1}$ of $\bG$.

%
%
The proof of Theorem~\ref{thm:iss_inclusion} is supported by three technical lemmas. The first is a routine variation of~\cite[Proposition 3.1]{sarkans2015input}, and the second and third are based on~\cite[Lemma 2.1.6, p.11]{gilmore_thesis} and~\cite[Lemma 4.1.11, p.\ 101]{gilmore_thesis}, respectively, although neither of these results apply in the set-valued case. As noted in~\cite[Remark 4.1.12, p.\ 101]{gilmore_thesis}, a version of the third lemma below appears inside the proof of~\cite[Theorem 1]{arcak2002input}.
%
%
\begin{lemma}
Consider~\eqref{eq:lure_inclusion} and assume that the pair~$(C,A)$ is detectable. Then there exists a symmetric positive definite~$Q \in \mR^{n\times n}$ and~$\delta >0$ such that, for all~$(u,z,w) \in \mR^{m} \times \mR^n \times \mR^m$,
\begin{equation}\label{eq:detectable_lyapunov}
2 \langle Q z, A z - B(w - u) \rangle \leq -\delta \| z \|^2 +  \| z \|\big(\| C z \| + \| w \| + \| u \|\big)\,. 
\end{equation} 
\end{lemma}

%
%
\begin{lemma}\label{lem:comp_function_for_}
Given~$\beta_i \in \sK_\infty$,~$i = 1,2,3$, and~$\mu \geq 0$, there exists~$\eta \in \sK_\infty$ such that
\begin{equation}\label{eq:comp_function_estimate_1}
\eta(\beta_1(s))\beta_2(s) \leq \beta_3(s) \quad \forall \: s \in [0, \mu]\,.
\end{equation} 
\end{lemma}
%
%
\begin{proof}
If~$\mu = 0$, then any~$\gamma \in \sK_\infty$ satisfies~\eqref{eq:comp_function_estimate_1}. So assume that~$\mu >0$. Defining~$\eta \in \sK_\infty$ by
\[ \eta(s) : = \frac{\beta_3(\beta_1^{-1}(s))}{\beta_2(\mu)} \quad \forall \: s \in \mR_+\,,\]
yields the required inequality, namely,
\[ \eta(\beta_1(s))\beta_2(s) = \beta_3(s) \frac{\beta_2(s)}{\beta_2(\mu)} \leq \beta_3(s) \quad \forall \: s \in [0, \mu]\,. \qedhere \]
\end{proof}

%
%
\begin{lemma}\label{lem:function_estimates}
Let~$F$ be given by~\eqref{eq:F_passive_inclusion} and assume that~$\alpha \in \sK_\infty$. The following bound holds:
\begin{equation}\label{eq:g_property_1}
2\langle u,y  \rangle \leq \langle y , w \rangle + 2 \alpha^{-1}(2 \| u \|) \| u\| \quad \forall \: w \in F(y), \; \forall \: (u,y) \in  \mR^n \times \mR^m\,.
\end{equation} 
Moreover, there exist~$\eps >0$ and~$\eta \in \sK_\infty$ such that
\begin{equation}\label{eq:g_property_2}
\eta(\| y \|) \| y \|^2 +  \eta(\|w \|) \|w\|^2 \leq \langle y , w \rangle \quad \forall \: w \in F(y), \; \forall \: y \in \mR^m, \; \| y\| \leq \mu\,,
\end{equation} 
and
\begin{equation}\label{eq:g_property_3}
\eps (\| w \| + \| y \|) \leq \langle y , w \rangle \quad \forall \: w \in F(y), \; \forall \: y \in \mR^m, \; \| y\| > \mu\,.
\end{equation} 
\end{lemma}
%
%
\begin{proof}
Note that, by construction, every pair~$(w,y)$ with~$y \in \mR^m$ and~$w \in F(y)$ satisfies~\eqref{eq:F_sector}. 
Let~$(u,y) \in  \mR^n \times \mR^m$ and~$w \in F(y)$. If~$\| u \| < \alpha(\| y \|)/2$, then the Cauchy-Schwarz inequality and~\eqref{eq:F_sector_1} give
\[ 2\langle y, u \rangle \leq 2 \| y \| \| u \| \leq \| y \| \alpha( \|y\|) \leq \langle y, w \rangle\,. \]
 If~$\| u \| \geq \alpha(\|y\|)/2$, then evidently
\[ 2\langle y, u \rangle \leq 2 \| y \| \| u \| \leq  2\| u \| \alpha^{-1}( 2\| u\|)\,. \]
Since both the terms on the right hand side of~\eqref{eq:g_property_1} are nonnegative, it follows that
\[ 2\langle y, u \rangle \leq \max\big\{\langle y, w \rangle, 2\| u \| \alpha^{-1}( 2\|u\|) \big\} \leq \langle y, w \rangle + 2 \alpha^{-1}(2 \| u \|) \| u\|\,,\]
as required. 
To prove~\eqref{eq:g_property_2}, recall first that~$\theta \in \sK_\infty$ is such that~$\| w\| \leq \theta(\| y \|)$ for all~$y \in \mR^m$. We apply Lemma~\ref{lem:comp_function_for_} with
\[ \beta_1(s) : = s + \theta(s), \quad \beta_2(s) := 2(s^2 + \theta^2(s)), \quad \text{and} \quad \beta_3(s) := s \alpha(s)\,,\]
to yield~$\eta \in \sK_\infty$ such that~\eqref{eq:comp_function_estimate_1} holds. We estimate that
\begin{align*}
\eta(\|y\|) \|y\|^2 + \eta(\|w\|) \| w\|^2 & \leq \eta(\|y\|) \|y\|^2 + \eta(\theta(\|y\|) \| \theta^2(\|y\|) \\
& \leq 2 \eta\big(\|y\| + \theta(\|y\|) \big)\big( \|y\|^2 + \theta^2(\|y\|)\big) \\
& = \eta(\beta_1( \| y \|))\beta_2( \| y \|) \leq \beta_3(\|y\|) \\
& = \|y \| \alpha(\| y \|) \quad \forall \: y \in \mR^m \quad \text{such that} \quad \| y \| \leq \mu\,.
\end{align*}
The inequality~\eqref{eq:g_property_2} now follows from the above and~\eqref{eq:F_sector_1}.

Finally, to prove~\eqref{eq:g_property_3}, we set
\[ \eps : = \min\Big\{\frac{1}{2c}, \frac{\alpha(\mu)}{2}\Big\} >0\,,\]
and use~\eqref{eq:F_sector_1} and~\eqref{eq:F_sector_2} to obtain the desired inequality
\begin{align*}
 \eps (\| w \| + \| y \|) & \leq \frac{1}{2c}\| w \|  + \frac{\alpha(\mu)}{2}\| y \| \leq \frac{1}{2c}\| w \|  + \frac{\alpha(\| y \|)}{2}\| y \| \leq \frac{1}{2}\langle y, w \rangle + \frac{1}{2 }\langle y, w \rangle\\
& = \langle y, w \rangle \quad \forall \: y \in \mR^m \quad \text{such that} \quad \| y \| > \mu\,.  \qedhere
\end{align*}
\end{proof}

%
%
\begin{proof}[Proof of Theorem~\ref{thm:iss_inclusion}]
We introduce some notation and estimates common to each of the statements. Let~$\eps >0$ and~$\eta \in \sK_\infty$ be as in Lemma~\ref{lem:function_estimates}. Define~$V_P : \mR^n \to \mR_+$ and~$V_Q : \mR^n \to \mR_+$ by the quadratic forms
\[ V_P(z) : = \langle z , P z \rangle \quad \text{and} \quad V_Q(z) : = \langle z , Q z \rangle \quad \forall \: z \in \mR^n\,, \]
where~$P$ and~$Q$ are as in~\eqref{eq:P_lmi_condition} and~\eqref{eq:detectable_lyapunov}, respectively. Since~$Q$ is positive definite, there exist~$q_1, q_2 >0$ such that
\begin{equation}\label{eq:Q_constants}
q_1 \| z \| \leq V_Q(z) \leq q_2 \| z \| \quad \forall \: z \in \mR^n\,.
\end{equation}
In the following, let~$(u,z) \in \mR^m \times \mR^n$ and~$w \in F(Cz)$. We compute that
\begin{align}
\langle (\nabla V_P)(z), Az - B(w - u)\rangle & = 2 \langle Pz, Az - B(w - u)\rangle \notag \\
& = \left \langle \bpm{A^\top P + PA & PB - C^\top \\ B^\top P - C & 0} \bpm{z \\ u-w}, \bpm{z \\ u-w} \right \rangle 
+ 2\langle u-w, C z \rangle \notag \\
& \leq -2\langle w, C z \rangle + 2 \langle u, C z \rangle\,, \label{eq:inclusion_iss_p1}
\end{align}
where we have used~\eqref{eq:P_lmi_condition}.

\ref{ls:inclusion_ISS} An application of the estimate~\eqref{eq:g_property_1} from Lemma~\ref{lem:function_estimates} to~\eqref{eq:inclusion_iss_p1} gives that
\begin{equation}\label{eq:iss_p1}
\langle (\nabla V_P)(z), Az - B(w - u)\rangle \leq -\langle w, Cz \rangle + 2 \alpha^{-1}(2 \| u \|) \| u\|\,.
\end{equation}
Second, we invoke~\eqref{eq:detectable_lyapunov} to estimate that
\begin{align}
\langle (\nabla V_Q)(z), Az - B(w - u)\rangle & = 2 \langle Qz, Az - B(w - u)\rangle \notag \\
& \leq -\delta \| z \|^2 +  \| z \|( \|Cz \| + \| w\|) + \| u \|^2\,. \label{eq:iss_p2} 
\end{align}
Next, we define~$W := h \circ V_Q$ where~$h :\mR_+ \to \mR_+$ is given by
\[ h(s) : = \int_0^p k(\tau)\, \rd  \tau \quad \forall \: s \geq0 \quad \text{and} \quad k(\tau) : = c_0 \min \Big\{ \frac{1}{\sqrt{\tau + 1}}, c_1\eta(c_2\sqrt{\tau})\Big\} \quad \forall \: \tau \geq 0\,, \]
where
\[ c_0 : = \min\big\{ 1, \eps/q_1\big\}, \quad c_1 := \frac{\delta}{4} \quad \text{and} \quad c_2 := \frac{\delta}{4  q_2}\,,\]
 and where~$q_1, q_2 >0$ are as in~\eqref{eq:Q_constants}, and~$\eta \in \sK_\infty$ is as in Lemma~\ref{lem:function_estimates}. 
Note that~$k$ is bounded, but~$h$ is unbounded, and hence~$W$ is radially unbounded. Note further that since~$k(\tau) = c_0/\sqrt{\tau+1}$ for sufficiently large~$\tau >0$, it follows that
\[ \frac{\delta  k(V_Q(z))\| z \|^2}{2} = \frac{\delta c_0\| z \|^2}{2\sqrt{\lvert z \rvert_Q^2 + 1}} \geq \frac{\delta c_0\|z\|}{4 q_2}\,,\]
for all~$z \in \mR^n$ with~$\| z \|$ sufficiently large. Therefore, it follows that there exists~$\gamma \in \sK_\infty$ such that
\begin{equation}\label{eq:iss_p3}
\frac{\delta k(V_Q(z))\| z \|^2}{2} \geq \gamma(\|z\|) \quad \forall \: z \in \mR^n\,.
\end{equation}
We proceed to compute that
\begin{align}
\langle (\nabla W)(z), Az - B(w - u)\rangle & \leq h'(V_Q(z)) \big( -\delta \| z \|^2 + \| z \|(\| C z \| + \|w\|) + \| u \|^2\big) \notag \\
& = k(V_Q(z))\big( -\delta \| z \|^2 + \| z \|(\| C z \| + \|w\|) + \| u \|^2\big) \notag \\
& \leq k(V_Q(z))\big(  -\delta \| z \|^2 + \| z \|(\| C z \| + \|w\|)\big) +  \| k \|_{L^\infty(\mR_+)}  \| u \|^2\,. \label{eq:iss_p4}
\end{align}
We claim that
\begin{equation}\label{eq:iss_cross_estimate}
k(V_Q(z)) \| z \|(\| C z \| + \|w\|) \leq \frac{\delta k(V_Q(z))}{2} \| z \|^2 + \langle w, C z\rangle\,. 
\end{equation}
Assuming that~\eqref{eq:iss_cross_estimate} holds, setting~$V: = V_P + W$ and summing~\eqref{eq:iss_p1} and~\eqref{eq:iss_p4} yields that
\begin{align*}
\langle (\nabla V)(z), Az - B(w - u)\rangle &\leq -\frac{\delta  k(V_Q(z))}{2} \| z \|^2 + \| k \|_{L^\infty(\mR_+)} \| u \|^2 + 2 \alpha^{-1}(2 \| u \|) \| u\| \\
& \leq -\gamma(\|z\|) + \| k \|_{L^\infty(\mR_+)} \| u \|^2 + 2 \alpha^{-1}(2 \| u \|) \| u\|\,, 
\end{align*}
by~\eqref{eq:iss_p3}. The function~$V$ is an ISS Lyapunov function for the Lur'e inclusion~\eqref{eq:lure_inclusion}, and so we conclude that~\eqref{eq:lure_inclusion} is ISS from Lemma~\ref{lem:inclusion_ISS_lyapunov}. 

It remains to establish~\eqref{eq:iss_cross_estimate}. We consider two exhaustive cases. 

{\sc Case 1:~$\| C z \| \leq \mu$}. If~$\|w \| \leq (\delta/4) \| z \|$, then trivially 
\[ k(V_Q(z)) \| z \|\|w\| \leq \frac{\delta}{4}k(V_Q(z))\|z \|^2\,.\]
Conversely,  if~$\|w \| \geq (\delta/4) \| z \|$, then
\begin{align*}
  k(V_Q(z)) \| z \|\|w\| & \leq \frac{4}{\delta}k(V_Q(z)) \|w \|^2 \leq \frac{4}{\delta} c_0 c_1 \eta\Big( c_2 \sqrt{V_Q(z)}\Big)  \| w \|^2 \leq \frac{4}{\delta} c_0 c_1 \eta( c_2 q_2 \| z \|) \| w\|^2 \\
	& \leq \frac{4}{\delta} c_0 c_1 \eta\big( (4 c_2 q_2/\delta)  \|w \| \big) \| w\|^2\leq  \eta(\| w\|)\| w\|^2\,,
\end{align*}
where we have used that~$k(s) \leq c_0 c_1 \eta\big(c_2 \sqrt{s}\big)$, that~$\eta$ is non-decreasing, and the definitions of the constants~$c_0$,~$c_1$ and~$c_2$.
Combining the above two exhaustive cases yields that
\begin{equation}\label{eq:p6}
 k(V_Q(z)) \| z \|\|w\| \leq \max\Big\{ \frac{\delta}{4}k(V_Q(z))\| z \|^2 ,  \eta(\| w\|)\| w\|^2\Big\} \leq \frac{\delta}{4}k(V_Q(z))\| z \|^2 +  \eta(\| w\|)\| w\|^2\,.
\end{equation} 
The same argument with~$\| w\|$ replaced by~$\|C z\|$ yields that 
\begin{equation}\label{eq:p7}
 k(V_Q(z)) \| z \|\| C z\|  \leq \frac{\delta}{4} k(V_Q(z)) \| z \|^2 +  \eta(\|Cz\|)\|Cz\|^2\,.
\end{equation}
Summing~\eqref{eq:p6} and~\eqref{eq:p7} gives
\begin{align}
 k(V_Q(z)) \| z \|(\| C z \| + \|w\|) 
& \leq \frac{\delta}{2}k(V_Q(z))\| z \|^2 + \eta(\| C z \|) \| Cz \|^2 +  \eta(\|w \|) \|w \|^2 \notag \\
& \leq \frac{\delta}{2}k(V_Q(z))\| z \|^2  + \langle w, C z \rangle\,, \label{eq:p8}
\end{align}
by~\eqref{eq:g_property_2}. 

{\sc Case 2:~$\| C z \| > \mu$}. Note that by construction of~$k$ it follows that
\[ z \mapsto k(V_Q(z)) \| z\| \quad \text{is bounded by~$\eps$}\,.\]
Consequently,
\begin{align}
 k(V_Q(z))  \| z \|(\| C z \| + \| w\|) & \leq  \eps \big(\| C z \| + \| w\|\big)  \leq    \langle w, C z \rangle  \leq \frac{\delta}{2}  k(V_Q(z))\| z \|^2 + \langle w, C z\rangle\,, \label{eq:p9}
\end{align}
by~\eqref{eq:g_property_3}. The conjunction of~\eqref{eq:p8} and~\eqref{eq:p9} establishes~\eqref{eq:iss_cross_estimate}, completing the proof of statement~\ref{ls:inclusion_ISS}. 

%
%
As shall become apparent, we prove statement~\ref{ls:inclusion_exp_ISS} next. For the proofs of statements~\ref{ls:inclusion_sg_exp_ISS} and~\ref{ls:inclusion_exp_ISS}, since~$\alpha$ appears as a lower bound in the estimate~\eqref{eq:F_sector_1}, it suffices to replace~$\alpha$ by~$\alpha(s) := \eps_0 s$. 

\ref{ls:inclusion_exp_ISS} With our previous assumption on~$\alpha$, the inequality~\eqref{eq:F_sector_1} becomes
%
\[ \eps_0^2 \| y \| \leq \langle y, w \rangle \quad \forall \: w \in F(y), \; \forall \: y \in \mR^m\,,\]
so that~\eqref{eq:inclusion_iss_p1} admits the estimate
\begin{equation}\label{ls:exp_iss_p1}
\langle (\nabla V_P)(z), Az - B(w - u)\rangle  \leq -\delta_1 \| C z\|^2 + \delta_2 \| u\|^2\,,
\end{equation}
for some suitable constants~$\delta_1, \delta_2 >0$.
 The linear boundedness assumption on~$\theta$ ensures the existence of~$L_0 > 0~$ such that
\[ \| w \| \leq \theta(\|y\|) \leq L_0 \| y \| \quad \forall \: w \in F(y), \; \forall \: y \in \mR^m\,.\]
Therefore, the estimate~\eqref{eq:iss_p2} becomes
\begin{equation}\label{ls:exp_iss_p2}
\langle (\nabla V_Q)(z), Az - B(w - u)\rangle  \leq -\delta_3 \| z\|^2 + \delta_4 \| C z \|^2 + \| u\|^2\,,
\end{equation}
now for some suitable~$\delta_3, \delta_4 >0$. In light of~\eqref{ls:exp_iss_p1} and~\eqref{ls:exp_iss_p2}, it is clear that
\[ V: = V_P + \frac{\delta_1}{\delta_4}V_Q\,,\]
satisfies
\[ \langle (\nabla V)(z), Az - B(w - u)\rangle  \leq -\delta_5 \| z\|^2 +  \delta_6 \| u\|^2\,, \]
for some suitable constants~$\delta_5, \delta_6 >0$. Consequently,~$V$ is an exponential ISS Lyapunov function, and the claim is proven via an application of Lemma~\ref{lem:inclusion_ISS_lyapunov}.

\ref{ls:inclusion_sg_exp_ISS} Let~$R >0$ be fixed, but arbitrary. Since~$\alpha(s) = \eps_0 s$ satisfies~$\alpha \in \sK_\infty$, it follows from statement~\ref{ls:inclusion_ISS} that there exists~$b >0$ such that for all~$(x(0) , v) \in \mR^n \times L^\infty(\mR_+,\mR^m)$ with~$\| x(0)\| + \| v\|_{L^\infty} \leq R$, 
\[ \| C x(t) \| + \| x(t) \| \leq b \quad \forall \: t \geq 0\,.\]
Since~$\theta$ is continuously differentiable with~$\theta(0) =0$, it is locally Lipschitz and, hence, there exists~$L_1 >0$ such that
\begin{equation}\label{eq:inclusion_sg_exp_iss_p1}
    \| w  \| \leq \theta(\|y \|) \leq L_1 \| y \| \quad \forall \: w \in F(y), \; \forall \: y \in {\rm cl}\,\mB(0,b)\,.
\end{equation}
Without loss of generality, we assume that~$b \geq \mu$ as in~\eqref{eq:F_sector_2}. 
Define~$\tilde F$ by
\[
    \tilde F(y) : = \big\{ w \in \mR^m \: : \: \| w \| \leq L_1 \|y\| \quad \text{and~\eqref{eq:F_sector} holds with~$\alpha$ a linear function}\big\}\,.
\]
Therefore, an application of statement~\ref{ls:inclusion_exp_ISS} gives that Lur'e inclusion~$(A,B,C,\tilde F)$ is exponentially ISS. However, by construction and~\eqref{eq:inclusion_sg_exp_iss_p1}, it follows that if~$(u,x) \in \sB_{\rm inc}^\infty(F)$ with~$\| x(0)\| + \| v\|_{L^\infty} \leq R$, then~$(u,x) \in \sB_{\rm inc}^\infty(\tilde F)$, as required.
\end{proof}

%
%
The next theorem guarantees the existence of certain stability properties and Lyapunov functions for an unforced case of Lur'e inclusion~\eqref{eq:lure_inclusion}. Here we are able to weaken the assumptions on~$F$ in~\eqref{eq:F_passive_inclusion}.  
\begin{theorem}\label{thm:KL_stability}
Consider the unforced Lur'e inclusion~$(A,B,C,F_0)$ where
\begin{equation}\label{eq:F_passive_inclusion_0}
F_0(y) : = \big\{ w \in \mR^m \: : \: \| w \| \leq \theta(\|y \|) \quad \text{and} \quad \langle w, y \rangle \geq \| y \| \alpha(\|y\|)\big\}\,,  
\end{equation}
for~$\theta \in \sK_\infty$,~$\alpha \in \sP$ with~$\theta \leq \alpha$, and where
\begin{equation}\label{eq:technical}
    \quad s \mapsto r(s) : = \sqrt{\theta^2(s) - \alpha^2(s)} \quad \text{is non-decreasing and locally Lipschitz.}
\end{equation}
Moreover, assume that the pair~$(C,A)$ is detectable, and that there exists a symmetric positive definite matrix~$P\in\mR^{n\times n}$ satisfying~\eqref{eq:P_lmi_condition}. The following statements hold.
\begin{enumerate}[label={\rm \bf (\arabic*)}, ref = {\rm \bf (\arabic*)}]
\item \label{ls:inclusion_GAS_1} The quadratic form~$V_P(z):=\langle Pz,z\rangle$ satisfies
\begin{equation}\label{eq:V_orbital_derivative}
\langle(\nabla V_P)(z),w\rangle\leq 0\quad\forall\:w\in Az -BF(Cz),\,\,\forall\: z\in\mR^n\,.
\end{equation}
Furthermore, for each compact set~$\Gamma\subset\mR^n$ such that~$\min_{z\in\Gamma}\|Cz\|>0$,
there exists~$\nu>0$ such that
\begin{equation}\label{eq:V_orbital_derivative<0}
\langle(\nabla V_P)(z),w\rangle\leq-\nu\quad\forall\:w\in Az-BF(Cz),\,\,\forall\:z\in\Gamma\,.
\end{equation}
\item \label{ls:inclusion_GAS_2} There exists~$\kappa>0$ such that~$\|x(t)\|\leq\kappa\|x(0)\|$ for all~$t\geq 0$ and all
$x\in W^{1,1}_{\rm loc}(\mR_+,\mR^n)$ such that~$(0,x)\in\sB_{\rm inc}$, that is,~$0$ is uniformly stable in the large. In particular, for every $v \in L^1(\mR_+,\mR^m)$, every solution $x$ of~\eqref{eq:lure_inclusion} may be extended to a global one.
\item \label{ls:inclusion_GAS_3} For every~$\eps>0$ and every~$\rho>0$, there exists~$\tau\geq 0$ such that, for each~$x\in W^{1,1}_{\rm loc}(\mR_+,\mR^n)$ with~$(0,x)\in\sB_{\rm inc}$, the following implication
holds
\[
\|x(0)\|\leq\rho\quad\Rightarrow\quad\|x(t)\|\leq\eps\,\,\,\forall\:t\geq\tau\,,
\]
that is,~$0$ is uniformly globally attractive.
\item \label{ls:inclusion_GAS_4} There exists a radially unbounded smooth function~$U:\mR^n\to\mR_+$ such that
$U(z)>0$ for all~$z \neq 0$ and
\begin{equation}\label{eq:orbital_derivative}
\langle(\nabla U)(z),w\rangle\leq-U(z)\quad\forall\:w\in Az-BF_0(Cz),\,\,\forall\:z\in\mR^n\,.
\end{equation}
The function~$U$ depends only on the linear data~$(A,B,C)$, and the bounding terms~$\alpha$ and~$\theta$.
\end{enumerate}
\end{theorem}
Note that in statement~\ref{ls:inclusion_GAS_4}, as~$F_0(0)=\{0\}$, we may conclude that~$U(0)=0$, and the hypotheses on~$U$ guarantee that there exist~$\alpha_1,\alpha_2\in\sK_\infty$ such that~$\alpha_1(\|z\|)\leq U(z)\leq\alpha_2(\|z\|)$
for all~$z\in\mR^n$. It follows from~\eqref{eq:orbital_derivative} that
\[
\langle(\nabla U)(z),w\rangle\leq-\alpha_1(\|z\|)\quad\forall\:w\in Az-BF_0(Cz),\,\,\forall\:z\in\mR^n\,.
\]
Consider~$F$ and~$F_0$ given by~\eqref{eq:F_passive_inclusion} and~\eqref{eq:F_passive_inclusion_0}, respectively, for the same~$\theta$ and~$\alpha$. Observe that~$F(y) \subseteq F_0(y)$ for all~$y \in \mR^m$. Consequently, if~$(0,x) \in \sB_{\rm inc}(F)$, then~$(0,x) \in \sB_{\rm inc}(F_0)$.

The hypothesis~\eqref{eq:technical} is used in our proof of Theorem~\ref{thm:KL_stability}. Since~$\theta$ and~$\alpha$ play the role of upper- and lower-bounds, respectively, it does not seem restrictive and can be satisfied by making~$\theta$ ``larger''. Indeed, replacing~$\theta$ by~$\theta + \alpha$ gives that~$r : \mR_+ \to \mR_+$ satisfies
\[ \theta(s) \leq r(s) = \sqrt{\theta^2(s) + 2\alpha(s) \theta(s)} \leq \sqrt{3} \theta(s) \quad \forall \: s \geq 0\,,\]
so that~$r$ is non-decreasing, continuously differentiable on~$(0,\infty)$ (when~$\alpha$ and~$\theta$ are) and bounded from above and below by locally Lipschitz functions, and hence is itself locally Lipschitz. Note that~$\theta$ does not play an explicit role in the hypotheses of Theorem~\ref{thm:KL_stability}, and rather is imposed to ensure uniformity in the various stability results. 

The next technical lemma gathers properties of~$F_0$. The proof of the lemma appears in the Appendix.
%
%
\begin{lemma}\label{lem:inclusion_F0_prep}
 Let~$\theta \in \sK_\infty$,~$\alpha \in \sP$  be given such that~$\theta \leq \alpha$ and~\eqref{eq:technical} holds. The correspondence~$F_0$ given by~\eqref{eq:F_passive_inclusion_0} takes non-empty-, convex- and compact-values, is upper hemicontinuous and locally Lipschitz.
\end{lemma}

%
%
\begin{proof}[Proof of Theorem~\ref{thm:KL_stability}]
\ref{ls:inclusion_GAS_1} Let~$P=P^\top \in\mR^{n\times n}$ be a positive definite solution of~\eqref{eq:P_lmi_condition}. Taking~$u = 0$ in~\eqref{eq:inclusion_iss_p1} gives
\begin{align}
\langle (\nabla V_P)(z), Az - Bw \rangle & \leq - 2\langle w, C z \rangle \leq - \| Cz \| \alpha(\| Cz\|) \leq 0\quad \forall\: w \in F_0(Cz), \; \forall \: z \in \mR^n\,,\label{eq:inclusion_prep_0}
\end{align}
from which the inequality~\eqref{eq:V_orbital_derivative} readily follows. Note that the inequality~\eqref{eq:F_sector_2} (omitted from~$F_0$) is not required for the inequality~\eqref{eq:inclusion_iss_p1}.

To establish~\eqref{eq:V_orbital_derivative<0}, let~$\Gamma\subset\mR^n$ be compact and such that
$\min_{z\in\Gamma}\|Cz\|>0$. Seeking a contradiction, suppose that the claim is not true, in which case
\eqref{eq:inclusion_prep_0} implies the existence of a sequence~$\big((z_j,w_j)\big)_{j}$ with~$z_j\in\Gamma$ and~$w_j\in F(Cz_j)$ such that
\begin{equation}\label{eq:contradiction}
\lim_{j\to\infty}\langle w_j, C z_j \rangle =0\,.
\end{equation}
As  the sequence~$\big((z_j,w_j)\big)_{j}$ is bounded, it has a convergent subsequence, the limit of which we denote by~$(z_\infty,w_\infty)$. Using the compactness of~$\Gamma$ and~$F(y)$ for all~$y\in\mR^m$ and the upper hemicontinuity of~$F$, we conclude that~$z_\infty\in\Gamma$ and~$w_\infty\in F(Cz_\infty)$. Consequently,~$Cz_\infty\not=0$, and, by~\eqref{eq:F_sector_1} as~$\alpha \in \sP$, 
\[ \langle w_\infty, C z_\infty \rangle \geq \| C z_\infty\| \alpha(\| C z_\infty\|) >0\,,\]
yielding a contradiction to~\eqref{eq:contradiction}.

\ref{ls:inclusion_GAS_2}  For~$x\in W^{1,1}_{\rm loc}(\mR_+,\mR^n)$ such that~$(0,x)\in\sB_{\rm inc}$
it follows from~\eqref{eq:V_orbital_derivative} that~$({\rm d} (V_P\circ x)/{\rm d}t)(t)\leq 0$
for almost every~$t\geq 0$, and so
\[
V_P(x(t))\leq V_P(x(0))\quad\forall\:t\geq 0.
\]
The claim now follows as the map~$z\mapsto\sqrt{V_P(z)}$ is a norm on~$\mR^n$.

\ref{ls:inclusion_GAS_3}  Let~$\eps,\rho>0$ be given. By statement (2), there exists~$\kappa>0$ such that,
\[
\| x(t)\| \leq \kappa \| x(0)\| \quad \forall\:t \geq 0,\,\,\forall\:
x\in W^{1,1}_{\rm loc}(\mR_+,\mR^n)\,\,\,\mbox{with~$(0,x)\in\sB_{\rm inc}$},
\]
and so,
\begin{equation}\label{eq:tilde_F_inclusion_p1}
\|Cx(t)\| \leq \kappa\rho \| C \| =: b \quad \forall\:t \geq 0,\,\,\forall\:x\in W^{1,1}_{\rm loc}(\mR_+,\mR^n)\,\,\,\mbox{with~$(0,x)\in\sB_{\rm inc}$ and~$\|x(0)\|\leq\rho$}.
\end{equation}
An application of~\cite[Lemma 18]{MR3245919} yields~$\alpha_1 \in \sK_\infty$ and~$\sigma \in \sL$ such that
\[
\alpha(s) \geq \alpha_1(s) \sigma(s) \quad\forall\:s\geq 0\,.
\]
Define~$\tilde \alpha \in \sK_\infty$ by 
\[ \tilde \alpha(s) = \left\{ \begin{aligned} &\sigma(b) \alpha_1(s)\,, & & s\leq b\,, \\ & \sigma(b) \alpha_1(b) + \Big(1 - \frac{b}{s}\Big)s  & & s>b\,. \end{aligned} \right.\]
Now define the correspondence~$\tilde F_0$ by
\[
\tilde F_0(y) : = \big\{ w \in \mR^m \: : \: \| w \| \leq \theta(\|y \|) \quad \text{and} \quad \langle w, y \rangle \geq \| y \| \tilde \alpha(\|y\|)\big\}\,,  
\]
which satisfies~$\alpha(s) \geq \tilde \alpha(s)$ on~$[0,b]$. Consequently, by construction, 
\begin{equation}\label{eq:tilde_F_inclusion_p2}
    F_0(y) \subseteq \tilde F_0(y) \quad \forall \: y \in {\rm cl}\,\mB(0,b)\,.
\end{equation} 
We claim that 
\begin{equation}\label{eq:tilde_F_inclusion_p3}
    \text{$0$ is uniformly globally attractive for the unforced Lur'e inclusion~$(A,B,C,\tilde F_0)$.}
\end{equation}
Assuming for now that~\eqref{eq:tilde_F_inclusion_p3} holds, let~$(0,x)\in\sB_{\rm inc}(F_0)$ with~$x\in W^{1,1}_{\rm loc}(\mR_+,\mR^n)$ and~$\|x(0)\|\leq\rho$. It follows from~\eqref{eq:tilde_F_inclusion_p1} and~\eqref{eq:tilde_F_inclusion_p2}, that~$(0,x)\in\sB_{\rm inc}(\tilde F_0)$. In particular, property~\eqref{eq:tilde_F_inclusion_p3} ensures the existence of~$\tau >0$ such that
\[ \| x(t) \| \leq \eps \quad \forall \: t \geq \tau\,,\]
establishing the original uniform attractivity claim, as required.

It remains to establish~\eqref{eq:tilde_F_inclusion_p3}. For which purpose, let~$\rho, \eps >0$ be given, and consider~$(0,x)\in\sB_{\rm inc}(\tilde F_0)$ with~$x\in W^{1,1}_{\rm loc}(\mR_+,\mR^n)$ and~$\|x(0)\|\leq\rho$. An application of statement~\ref{ls:inclusion_GAS_2} to Lur'e inclusion~$(A,B,C,\tilde F_0)$ gives that~$x$ is bounded, and hence~$\dot x$ is bounded from set membership of a bounded set. Consequently,~$x$ is uniformly continuous. Invoking the inequality~\eqref{eq:inclusion_prep_0} with~$\alpha$ replaced by~$\tilde \alpha$, and integrating between~$0$ and~$t\geq 0$, gives that 
\begin{equation}\label{eq:F_tilde_UA_p1}
  \int_0^t \beta(\| Cx(s) \|) \, \rd s = \int_0^t \| C x(s) \| \tilde \alpha\big(\| C x(s)\|\big) \, \rd s \leq V_P(x(0)) \leq \| P \| \rho^2 \quad \forall \: t \geq 0\,, 
\end{equation}
where~$s \mapsto \beta(s) : = s \tilde \alpha(s) \in \sK_\infty$.  We claim that, for all~$\eps_0>0$, there exists~$\tau_0 >0$ such that 
\begin{equation}\label{eq:F_tilde_UA_p2}
\| Cz(t) \| \leq \eps_0 \quad \forall \: t \geq \tau_0\,. 
\end{equation}
Seeking a contradiction, suppose that this claim is false. In particular, there exists~$\eps^* >0$ and a non-decreasing real sequence~$(t_k)_k$ with~$t_k \geq k$ and such that
\[\| Cx(t_k) \| \geq \eps^* \quad \forall \: k \in \mN\,. \]
We may assume that~$\inf_{k \in \mN} \lvert t_{k+1} - t_k \rvert >0$. By uniform continuity of~$Cx$, there exists~$\delta >0$ such that
\[ \| Cx (t) \| \geq \frac{\eps^*}{2} \quad \forall \: t \in[t_k-\delta, t_k + \delta] \quad \forall \: k \in \mN\,,\]
and so
\[ \beta(\| Cx(t) \|) \geq \beta(\eps^*/2) >0 \quad \forall \: t \in[t_k-\delta, t_k + \delta] \quad \forall \: k \in \mN\,.\]
However, the above lower bound contradicts~\eqref{eq:F_tilde_UA_p1}.

Let~$H \in \mR^{n \times m}$ be such that~$A-HC$ is Hurwitz, the existence of which follows from the detectability hypothesis. Observe that the Lur'e inclusion~$(A,B,C,\tilde F_0)$ with~$v = 0$ may be expressed as
\begin{equation}\label{eq:F_tilde_UA_p3} 
\dot x(t) - (A-HC)x(t) - HCx(t) \in -B\tilde F_0(Cx(t)) \quad t \geq 0\,.
\end{equation}
The variation of parameters formula for the Lur'e inclusion~\eqref{eq:F_tilde_UA_p3} (see, for example~\cite[Lemma 2.1]{guiver2019small}) gives
\begin{equation}\label{eq:F_tilde_UA_p4}
    x(t+\tau_0) - \e^{(A - HC)t}x(\tau_0) \in  \int_{\tau_0}^{t+\tau_0} \e^{(A-HC)(t+\tau_0-s)}\big(HCx(s) - B\tilde F_0(Cx(s))\, \rd s \quad \forall \: t \geq 0\,.
\end{equation} 
Here recall that if~$z : \mR_+ \to \mR^n$ is a locally absolutely continuous solution of the differential inclusion~$\dot z(t) \in H(t,z(t))$ for correspondence~$H : \mR_+ \times \mR^n \rightrightarrows \mR^n$, then~$\dot z$ is a locally integrable selection of~$H(t,z(t))$, so that 
\[ \int_{t_1}^{t_2} H(s,z(s))\, \rd s := \Big\{ \int_{t_1}^{t_2} h(s) \, \rd s \: : \: h \in L^1(t_1,t_2, \mR^n) \; \text{is a selection of~$t \mapsto H(t,z(t))$}\Big\} \neq \emptyset\,,\]
and, with this notation,
\[ z(t_1) - z(t_2) \in \int_{t_1}^{t_2} H(s,z(s))\, \rd s \quad \forall \: t_2 \geq t_1 \geq 0\,.\]
Routine estimates of~\eqref{eq:F_tilde_UA_p4}, invoking~\eqref{eq:F_tilde_UA_p2}, that~$A-HC$ is Hurwitz, and that~$\|w \| \leq \theta(\|Cz(t)\|)$ for all~$w \in \tilde F_0(Cz(t))$, now show that 
\[ \|x(t + \tau_0) \| \leq M \e^{-\lambda t} \rho + c_0 M \frac{1}{\lambda}(1 - \e^{-\lambda t})\big (\eps_0 + \theta(\eps_0)\big) \quad \forall \: t \geq 0\,,\]
for some constants~$M, \lambda, c_0 >0$. The right-hand side of the above is no greater than~$\eps$ once~$\eps_0 >0$ is fixed sufficiently small, for all~$t \geq \tau_1$, for some~$\tau_1 >0$. The proof is complete.

\ref{ls:inclusion_GAS_4}  The conclusions of statements~\ref{ls:inclusion_GAS_1}--\ref{ls:inclusion_GAS_3} apply to the differential inclusion~\eqref{eq:lure_inclusion} and so, by~\cite[Proposition 1]{teel2000smooth},~\eqref{eq:lure_inclusion} is~$\sK\sL$-stable~\cite[Definition 6]{teel2000smooth}.

According to~\cite[Theorem 1]{teel2000smooth}, the existence of a claimed function~$U$ is equivalent to the
differential inclusion~\eqref{eq:lure_inclusion} being robustly~$\sK\sL$-stable in the sense of~\cite[Definition 8]{teel2000smooth}. Since~\eqref{eq:lure_inclusion} is~$\sK\sL$-stable, we can use~\cite[Theorem 8]{teel2000smooth} to establish robust~$\sK\sL$-stability by showing that~$F_0$ is locally Lipschitz in the sense of~\cite{teel2000smooth}. However, that~$F_0$ is locally Lipschitz was established in Lemma~\ref{lem:inclusion_F0_prep}. The proof is complete.
\end{proof}

%
%
The next result invokes the previous theorem to establish an IISS property for the Lur'e inclusion~\eqref{eq:lure_inclusion}.
\begin{theorem}\label{thm:iiss_inclusion}
Consider the Lur'e inclusion~$(A,B,C,F_0)$ where~$F_0$ is given by~\eqref{eq:F_passive_inclusion_0} for~$\theta \in \sK_\infty$ and~$\alpha \in \sP$ which sataisfy~\eqref{eq:technical}. Assume further that the pair~$(C,A)$ is detectable, and that there exists a positive definite matrix~$P=P^\top\in\mR^{n\times n}$ satisfying~\eqref{eq:P_lmi_condition}. Then~\eqref{eq:lure_inclusion} is IISS with linear IISS gain, and the IISS comparison functions depend only on the model data~$(A,B,C)$, and the bounding terms~$\theta$ and~$\alpha$.  
\end{theorem}
Observe that the hypotheses of Theorem~\ref{thm:iiss_inclusion} are the same as those of Theorem~\ref{thm:KL_stability}. In other words, conditions which are sufficient for robust~$\sK\sL$-stability of the unforced Lur'e inclusion~$(A,B,C,F_0)$  are additionally sufficient for the {\em a priori} stronger stability property of IISS of the forced version. The proof of Theorem~\ref{thm:iiss_inclusion} is the same as that of~\cite[Theorem 2.5]{gilmore2021incremental}, which itself is based on a combination of arguments from~\cite[Proposition 3.10]{MR4023134} and~\cite[Lemma 3.4]{MR2599348}, only replacing references to~\cite[Theorem 2.4]{gilmore2021incremental} by references to Theorem~\ref{thm:KL_stability}. Therefore, we omit the details.

%
%
The results so far have focussed on the role of the correspondence in the Lur'e inclusion in ensuring various stability notions. Our next result demonstrates that the assumptions on the correspondence may be weakened if, instead, stronger assumptions are imposed on the linear components of~\eqref{eq:lure_inclusion}.

\begin{proposition}\label{prop:inclusion_exp_ISS}
 Consider the Lur'e inclusion~$(A,B,C,H_0)$, where~$H_0$ is given by
 \begin{equation}\label{eq:F_passive_inclusion'}
 H_0(y) := \big \{ w \in \mR^m \: : \: \langle w, y \rangle \geq 0\big\}\,.
 \end{equation}
 If there exist symmetric positive definite~$P \in \mR^{n \times n}$ and~$\eps >0$ such that
\begin{equation}\label{eq:P_lmi_condition'}
\bpm{A^\top P + PA & PB - C^\top \\ B^\top P - C & 0 } \leq - \eps \bpm{I & 0 \\ 0 & 0}\,,
\end{equation}
then~\eqref{eq:lure_inclusion} is exponentially ISS. The exponential ISS constants depend only on the model data~$(A,B,C)$. 
\end{proposition}

Note that the correspondence~$H_0$ in~\eqref{eq:F_passive_inclusion'} is not compact valued. We do not focus on existence of solutions to the Lur'e inclusion~$(A,B,C,H_0)$ --- rather we shall see that solutions arise naturally when considering certain Lur'e differential equations in Section~\ref{sec:incremental}. By way of comparison of~$H_0$ and~$F_0$ in~\eqref{eq:F_passive_inclusion_0}, we see that~$H_0$ satisfies condition~\eqref{eq:F_sector_1} with~$\alpha = 0$, rather than with~$\alpha \in \sP$ for~$F_0$, and conditions~\eqref{eq:F_sector_0} and~\eqref{eq:F_sector_2} are absent from the definition of~$H_0$. In terms of uniformity, the exponential ISS constants evidently also depend on~$P$ and~$\eps$, but these are determined (only) by the linear data~$(A,B,C)$ via~\eqref{eq:P_lmi_condition'}.

Observe that condition~\eqref{eq:P_lmi_condition'} enforces 
\[ A^\top P + PA \leq -\eps I\,,\]
so that~$A$ is necessarily Hurwitz since~$P$ is assumed positive definite. Continuing the discussion from after Theorem~\ref{thm:iiss_inclusion} on the existence of solutions to~\eqref{eq:P_lmi_condition}, we comment here that if~$(A,B,C)$ is controllable and observable, and positive~$\eps>0$ is such that~$\bG_\eps := \bG(\cdot - \eps)$ is positive real (in other words,~$\bG$ is strictly positive real in the terminology of~\cite{guiver2017transfer}), then a positive definite solution~$P = P^\top$ to the LMI~\eqref{eq:P_lmi_condition'} again follows from the Positive Real Lemma, now applied to the triple~$(A+\eps I, B,C)$.
\begin{proof}[Proof of Proposition~\ref{prop:inclusion_exp_ISS}]
Define, as usual,~$V_P(z) : = \langle z, P z\rangle$ for all~$z \in \mR^n$, which satisfies
\begin{equation}\label{eq:inclusion_exp_iss_p1}
    c_1 \| z \|^2 \leq V_P(z) \leq c_2 \| z\|^2 \quad \forall \: z \in \mR^n\,,
\end{equation} 
for some~$c_1, c_2 >0$. 
The argumentation leading to inequality~\eqref{eq:inclusion_iss_p1} now becomes
    \begin{align}
\langle (\nabla V_P)(z), Az - B(w - u)\rangle & = 2 \langle Pz, Az - B(w - u)\rangle \notag \\
& = \left \langle \bpm{A^\top P + PA & PB - C^\top \\ B^\top P - C & 0} \bpm{z \\ u-w}, \bpm{z \\ u-w} \right \rangle 
+ 2\langle u-w, C z \rangle \notag \\
& \leq -\eps \| z \|^2 + 2 \langle u, C z \rangle \notag \\
& \leq -\eps_0 \| z \|^2 + \eps_1 \| u\|^2 \quad \forall \: w \in H_0(Cz), \; \forall \: (u,z) \in \mR^m \times \mR^n\,, \label{eq:inclusion_exp_iss_p2}
\end{align}
for some~$\eps_0, \eps_1 >0$, by a standard quadratic inequality. In light of the inequalities in~\eqref{eq:inclusion_exp_iss_p1} and~\eqref{eq:inclusion_exp_iss_p2}, the claimed exponential ISS property follows from Lemma~\ref{lem:inclusion_ISS_lyapunov}.
\end{proof}

\section{A passivity theorem for semi-global incremental stability of forced Lur'e differential equations}\label{sec:incremental}

Here we apply the results of Section~\ref{sec:inclusion} to derive semi-global incremental stability properties of the following system of forced Lur'e differential equations
\begin{equation}\label{eq:lure}
\dot x(t)=Ax(t)-Bf(t,Cx(t))+Bv(t) \quad t \geq 0\,,
\end{equation}
where~$(A,B,C)\in\mL$ and the time-varying nonlinearity~$f:\mR_+\times\mR^m\to\mR^m$ is assumed to have the following properties: 
\begin{itemize}
    \item the function~$t\mapsto f(t,y)$ is in~$L^1_{\rm loc}(\mR_+,\mR^m)$ for all~$y \in \mR^m$ (in particular,~$f$ is measurable in its first variable, for each fixed second variable);
    \item for every compact set~$\Gamma\subset\mR^m$, there exists~$\lambda\in L^1_{\rm loc}(\mR_+,\mR_+)$ such that
$\|f(t,y_1)-f(t,y_2)\|\leq\lambda(t)\|y_1-y_2\|$ for all~$y_1,y_2\in\Gamma$ and all~$t\geq 0$,
and;
    \item~$t \mapsto f(t,y_0)$ is essentially bounded for some~$y_0 \in \mR^m$.
\end{itemize}
 The requirement that~$t \mapsto f(t,y_0)$ is essentially bounded at a single point~$y_0$ is mild, as typically~$f(t, 0 ) =0$ for almost all~$t \in \mR_+$, which satisfies this assumption. Observe that when~$f$ is independent of its first variable, then the above requirements are trivially satisfied if~$f$ is locally Lipschitz. We will sometimes refer to~\eqref{eq:lure} as the Lur'e equation~$(A,B,C,f)$.

Let~$v\in L^1_{\rm loc}(\mR_+,\mR^m)$ be given and let~$0<\tau\leq\infty$. A function~$x\in W^{1,1}_{\rm loc}([0,\tau),\mR^n)$ is said to be a solution of~\eqref{eq:lure} on~$[0,\tau)$ if~\eqref{eq:lure} holds for almost every~$t\in[0,\tau)$. The assumptions on~$f$ are known to sufficient for the existence of unique solutions of the initial value problems associated with~\eqref{eq:lure}. A solution of~\eqref{eq:lure} on~$[0,\infty)=\mR_+$ is called a global solution.

We define two behaviours of~\eqref{eq:lure} by
\begin{align*}
  \sB &:= \big\{(v,x)\in L^1_{\rm loc}(\mR_+,\mR^m)\times
    W^{1,1}_{\rm loc}(\mR_+,\mR^n)\; : \; \text{$(v,x)$ satisfies~\eqref{eq:lure}
a.e. on~$\mR_+$}\big\} \\
\text{and} \quad \sB^\infty &:=\big\{ (v,x)\in  \sB_{\rm inc}  \; : \; v\in L^\infty_{\rm loc}(\mR_+,\mR^m) \big\}\,.
\end{align*}
%
%
The following two results are corollaries of the ISS results in Section~\ref{sec:inclusion} for the Lur'e inclusion~\eqref{eq:lure_inclusion}. The first and second place stronger assumptions on the nonlinear and linear data in~\eqref{eq:lure}, respectively. As we shall see, for practically interesting examples to fall within the scope of our development, we must impose ``semi-global'' incremental hypotheses, and obtain ``semi-global'' results. 

We record the following assumptions, noting that we need not impose every hypothesis simultaneously:
\begin{enumerate}[label = {\rm \bfseries (A\arabic*)}]
    \item \label{ls:A1} For every compact set~$\Gamma \subset \mR^m$, there exists~$\theta_\Gamma \in \sK_\infty$ such that
 \[
 \sup_{t \geq 0} \| f(t,y+z) - f(t,z) \| \leq \theta_\Gamma(\|y\|) \quad \forall \: y \in \mR^m, \; \forall \: z \in \Gamma\,. 
 \]
    %
    \item \label{ls:A2} For every compact set~$\Gamma \subset \mR^m$, there exists~$\alpha_\Gamma \in \sP$ such that 
    \begin{equation}\label{eq:f_sector_1'}
\| y \| \alpha_\Gamma(\| y \|) \leq \inf_{t \geq 0} \langle y, f(t,y + z) - f(t,z) \rangle \quad \forall \: y \in \mR^m, \; \forall \: z \in \Gamma \,. 
\end{equation}

    %
    \item \label{ls:A3} Hypothesis~\ref{ls:A2} holds with~$\alpha_\Gamma \in \sK_\infty$.
\item \label{ls:A4} For every compact set~$\Gamma \subset \mR^m$, there exist~$\mu_\Gamma, c_\Gamma \geq 0$ with~$\mu_\Gamma c_\Gamma \geq 1$, and such that
\begin{align*}
 \| f(t,y+z) - f(t,z) \| \leq c_\Gamma \langle y , &\,  f(t,y+z) - f(t,z) \rangle \notag \\
 & \quad \quad \forall \: y \in \mR^m, \; \| y \| > \mu_\Gamma, \; \forall \: z \in \Gamma, \; \forall \: t \geq 0\,. 
\end{align*}
\item \label{ls:A5} Hypothesis~\ref{ls:A2} holds with~$\alpha_\Gamma(s) \geq \eps_\Gamma s$, for some~$\eps_\Gamma >0$. 
\end{enumerate}

In light of our standing assumption that~$t \mapsto f(t,y_0)$ is essentially bounded for some~$y_0 \in \mR^m$, an immediate consequence of~\ref{ls:A1} is that the function~$t \mapsto f(t,y)$ is essentially bounded for every~$y \in \mR^m$. Indeed, let~$y \in \mR^m$ and define~$\Gamma := \{y_0\}$. Then, invoking~\ref{ls:A1}, we have for almost all~$t \geq0$,
\[ \| f(t,y) \| \leq \| f(t,(y - y_0)+ y_0) - f(t,y_0) \| + \| f(t,y_0)\|  \leq \theta_0(\|y-y_0\|) + \| f(t,y_0)\|\,, \]
the right-hand side of which is essentially bounded.

As with~\eqref{eq:F_sector_1} and~\eqref{eq:F_sector_2}, if~$m=1$ and~\eqref{eq:f_sector_1'} holds, then so does~\ref{ls:A4} with~$\mu_\Gamma = c_\Gamma =1$.

%
%
\begin{remark}\label{rem:technical}
Since there is no loss of generality in doing so, in the sequel whenever~\ref{ls:A1} and~\ref{ls:A2} are assumed to hold, the resulting~$\theta_\Gamma$ and~$\alpha_\Gamma$ are assumed to satisfy~$\alpha_\Gamma \leq \theta_\Gamma$ and~\eqref{eq:technical}, see the discussion after the statement of Theorem~\ref{thm:KL_stability}. 
\end{remark}

%
%
The next lemma provides sufficient conditions for hypotheses~\ref{ls:A2} and~\ref{ls:A3} in the simpler setting that~$f$ is independent of its first variable. The proof  appears in the Appendix.

\begin{lemma}\label{lem:sg_technical_preparation}
Let~$f(t,z) = f(z)$ as in~\eqref{eq:lure} be a function of its second variable only, and assume that~$f$ is locally Lipschitz. Let~$\Gamma \subset \mR^m$ be a compact set. The following statements hold. 
\begin{enumerate}[label = {\rm (\roman*)}]

    \item \label{ls:sg_technical_preparation_i} 
   The conditions
   \begin{equation}\label{eq:A2_sufficient}
       \left. \begin{aligned}
    \inf_{z \in \Gamma} \frac{\big \langle y, f(y + z) - f(z) \big\rangle }{\|y\|} &>0 \quad  \forall \: y \in \mR^m \backslash \{0\}\,, \;\\
     \text{and} \quad \lim_{\| y \|\searrow 0}\inf_{z \in \Gamma} \frac{\big \langle y, f(y + z) - f(z) \big\rangle }{\|y\|} &=0\,,
    \end{aligned}\right\}
   \end{equation}
    are together sufficient for hypothesis~\ref{ls:A2}.
\item \label{ls:sg_technical_preparation_ii} The conditions~\eqref{eq:A2_sufficient} combined with the existence of~$z_0 \in \mR^m$ such that
   \begin{equation}\label{eq:A3_sufficient}
   \frac{\big \langle y, f(y+z_0) - f(z_0) \big\rangle }{\| y \|} \to\infty \quad \text{as~$\|y \| \to\infty$,}
   \end{equation}
are together sufficient for hypothesis~\ref{ls:A3}.
\end{enumerate}
\end{lemma}
The following theorem contains sufficient conditions for a range of semi-global incremental stability properties for the Lur'e equation~\eqref{eq:lure}, and is the main result of this section.
%
%
\begin{theorem}\label{thm:incremental_passive_SG}
    Consider the Lur'e equation~\eqref{eq:lure}, assume that the pair~$(C,A)$ is detectable, and that there exists a symmetric positive semi-definite~$P \in \mR^{n\times n}$ which satisfies~\eqref{eq:P_lmi_condition}. Fix~$R>0$. The following statements hold.
    \begin{enumerate}[label={\rm \bf (\arabic*)}, ref = {\rm \bf (\arabic*)}]
    \item \label{ls:SG_IISS} If~$P$ is positive definite, and~\ref{ls:A1} and~\ref{ls:A2} hold, then there exist~$\psi \in \sK\sL$,~$\phi \in \sK$ such that, for every~$(w,z) \in \sB$ with~$\| z\|_{L^\infty} \leq R$, 
\begin{equation}\label{eq:incremental_IISS_lure}
    \|(x-z)(t)\|\leq\psi\big(\|(x-z)(0)\|,t\big)+\phi\Big(\int_0^t\|(v-w)(s)\|{\rm d}s\Big)\quad\forall\:t\geq 0,\,\,\forall\:
(v,x)\in\sB\,.
\end{equation}
\item \label{ls:SG_ISS}  If~\ref{ls:A1}, \ref{ls:A3}, and~\ref{ls:A4} hold, then there exist~$\psi \in \sK\sL$,~$\phi \in \sK$ such that, for every~$(w,z) \in \sB^\infty$ with~$\| z(0)\| + \| w \|_{L^\infty} \leq R$, 
\begin{equation}\label{eq:incremental_ISS_lure} 
\|(x-z)(t)\|\leq\psi\big(\|(x-z)(0)\|,t\big)+\phi(\|v - w\|_{L^\infty(0,t)}) \quad\forall\:t\geq 0,\,\,\forall\:
(v,x)\in\sB^\infty\,.
\end{equation}
\item \label{ls:SG_exp_ISS} If~\ref{ls:A1}, \ref{ls:A4} and~\ref{ls:A5} hold, then there exist~$M, \gamma >0$ such that, for every~$(v_i,x_i) \in \sB^\infty$ with~$\| x_i(0)\| + \| v_i \|_{L^\infty} \leq R$, 
\begin{equation}\label{eq:lure_incrementally_exp_iss}
    \|(x_1-x_2)(t)\|\leq M \big( \e^{-\gamma t} \|(x_1-x_2)(0)\| + \|v_1 -v_2\|_{L^\infty(0,t)} \big) \quad\forall\:t\geq 0\,.
\end{equation} 
    \end{enumerate}
\end{theorem}
%
%
\begin{proof}
Let~$(w,z), (v,x)  \in \sB$ denote two trajectories of~\eqref{eq:lure}, and set~$\xi := x - z$. Forming the difference and invoking~\eqref{eq:lure} gives
\begin{align}
\dot \xi(t) - A \xi(t) - B(v - w)(t)& = - B\big(f(t,C\xi(t) + Cz(t)) - f(t,Cz(t)) \big)  \notag \\
& \in -BG(C\xi(t))\,, \label{eq:lure_difference}
\end{align}
that is,~$(v-w,\xi) \in \sB_{\rm inc}(G)$, where~$G = F_0$ as in~\eqref{eq:F_passive_inclusion_0} if~\ref{ls:A1} and~\ref{ls:A2} are imposed, or $G = F$ if~\ref{ls:A1}, \ref{ls:A1} and~\ref{ls:A4} are imposed.
Indeed, if~$f$ satisfies~\ref{ls:A1} and~\ref{ls:A2}, then (via~\eqref{eq:f_sector_1'}), we have that
\[ f(t,Cx(t)) - f(t,Cz(t)) = f(t,C\xi(t) + Cz(t)) - f(t,Cz(t)) \in F_0(C\xi(t)) \quad \forall \: t \geq 0\,.\]
With the above construction, statement~\ref{ls:SG_IISS} follows from Theorem~\ref{thm:iiss_inclusion}.

For statement~\ref{ls:SG_ISS}, assumptions~\ref{ls:A1}, \ref{ls:A3} and~\ref{ls:A4}, each with~$\Gamma := \{0\}$, respectively yield
\begin{equation}
\| f(t,y) - f(t,0) \| \leq \theta_0(\|y\|) \quad \forall \: y \in \mR^m, \; \forall \: t \geq 0\,, \label{eq:f_sector_0''}
\end{equation}
and
   \begin{equation}\label{eq:f_sector_1''}
\| y \| \alpha_0(\| y \|) \leq \langle y, f(t,y) - f(t,0) \rangle \quad \forall \: y \in \mR^m, \; \forall \: t \geq 0\,,
\end{equation}
and
\begin{align}
 \| f(t,y) - f(t,0) \| \leq c_0 \langle y , &\,  f(t,y) - f(t,0) \rangle \quad \forall \: y \in \mR^m, \; \| y \| > \mu_0, \; \forall \: t \geq 0\,,\label{eq:f_sector_2''}
\end{align}
for some~$\alpha_0, \theta_0 \in \sK_\infty$, and~$\mu_0, c_0  >0$.

Now if~$(w,z) \in \sB^\infty$, then
\begin{align*}
    \dot z(t) &= Az(t) - Bf(t,Cz(t) + Bw(t) = Az(t) - B\big(f(t,Cz(t)) - f(t,0)\big) + B(w(t) -f(t,0)) \notag \\
    & \in Az(t) - BF(Cz(t)) + B(w(t)-f(t,0))\,, 
\end{align*}
that is,~$(w- f(\cdot, 0), z) \in \sB_{\rm inc}^\infty(F)$. The Lur'e inclusion~$(A,B,C,F)$ is ISS by statement~\ref{ls:inclusion_ISS} of Theorem~\ref{thm:iss_inclusion}, where~$F$ is given by~\eqref{eq:F_passive_inclusion}. Here we have used that properties~\eqref{eq:f_sector_0''}--\eqref{eq:f_sector_2''} ensure the set membership
\[ f(t,Cz(t)) - f(t,0) \in F(Cz(t)) \quad \forall \: t \in \mR_+\,,\]
In particular, there exists~$b >0$ such that~$\| z \|_{L^\infty} \leq b$ for all~$(w,z)\in \sB^\infty$ with~$\|z(0)\| + \|w\|_{L^\infty} \leq R$. Here we have used that~$f(\cdot, 0)$ is essentially bounded. We now basically repeat the above argument with~$\Gamma := {\rm cl}\,\mB(0,b)$, to obtain new functions/constants~$\theta_\Gamma$,~$\alpha_\Gamma$,~$\mu_\Gamma$ such that~\eqref{eq:f_sector_0''}--\eqref{eq:f_sector_1''} above hold, and apply Theorem~\ref{thm:iss_inclusion} again.

Statement~\ref{ls:SG_exp_ISS} follows from statement~\ref{ls:SG_ISS} and statement~\ref{ls:inclusion_sg_exp_ISS} of Theorem~\ref{thm:iss_inclusion}.
\end{proof}

%
%
We next state a corollary of Proposition~\ref{prop:inclusion_exp_ISS} for the Lur'e equation~\eqref{eq:lure}. As with Proposition~\ref{prop:inclusion_exp_ISS}, the corollary shows that weaker monotonicity assumptions may be placed on the nonlinearity~$f$ if stronger assumptions are instead imposed on the linear data.

\begin{corollary}\label{cor:exp_ISS_lure}
  Consider the Lur'e equation~\eqref{eq:lure} and assume that
  \begin{equation}\label{eq:monotone_f}
      \langle y,  f(t,y+z) - f(t,z) \rangle \geq 0 \quad \forall \: t \geq0, \; \forall \: z,y \in \mR^m\,.
  \end{equation} 
 If there exist symmetric positive definite~$P \in \mR^{n \times n}$ and~$\eps >0$ such that~\eqref{eq:P_lmi_condition'} holds, then the Lur'e equation is incrementally exponentially ISS, that is, there exist~$M, \gamma >0$ such that~\eqref{eq:lure_incrementally_exp_iss} holds for all~$(v_i,x_i) \in \sB^\infty$.
\end{corollary}
%
%
\begin{proof}
   Let~$(v_i,x_i) \in \sB$ for $i = 1,2$, and set~$x := x_1 - x_2$. Forming the difference and invoking~\eqref{eq:lure} gives
\begin{align*}
\dot x(t) - A x(t) - B(v_1 - v_2)(t) & =- B\big(f(t,C x(t) + Cx_2(t)) - f(t,Cx_2(t)) \big)  \notag \\
& \in - BH_0(Cx(t))\,,
\end{align*}
where~$H_0$ is as in~\eqref{eq:F_passive_inclusion'}. Here we have used~\eqref{eq:monotone_f} to ensure that
\[ f(t,Cx_1(t)) - f(t,Cx_2(t)) = f(t,Cx(t) + Cx_2(t)) - f(t,Cx_2(t)) \in H_0(Cx(t)) \quad \forall \: t \geq 0\,.\]
The claim now follows from Proposition~\ref{prop:inclusion_exp_ISS}.
\end{proof}

Some remarks on the hypotheses of Theorem~\ref{thm:incremental_passive_SG} and Corollary~\ref{cor:exp_ISS_lure} are in order. The assumptions on the linear components of the Lur'e equation~\eqref{eq:lure} are the same as those imposed for the Lur'e inclusion~\eqref{eq:lure_inclusion}, and the commentary after the statement of Theorem~\ref{thm:iss_inclusion} applies in this setting as well. 
%
%
Therefore, we focus attention on hypotheses~\ref{ls:A2}, \ref{ls:A3} and~\ref{ls:A5}. For which purpose, consider first the related inequality 
\[  \langle z_1 -z_2, f(t,z_1) - f(t,z_2) \rangle \geq 0 \quad \forall \: z_1, z_2 \in \mR, \; \forall \: t \in \mR_+\,,\]
which is simply monotonicity of~$f(t, \, \cdot)$; see for example~\cite[p.\ 97]{deimling2010nonlinear}. In the scalar ($m=1$) case, this is equivalent to the lower slope restriction condition
\[ 0 \leq \frac{f(t,z_1) - f(t,z_2) }{z_1 - z_2} \quad \forall \: z_1, z_2 \in \mR, \; z_1 \neq z_2, \; \forall \: t \in \mR_+\,,\]
and is also equivalent to~$\partial f/ \partial z \geq 0$ when~$f$ is differentiable with respect to its second variable. Slope restriction conditions are somewhat common in nonlinear control theory; see, for example~\cite{haddad1995absolute,turner2019analysis}. Roughly speaking, condition~\eqref{eq:f_sector_1'} with~$\alpha \in \sP$ or~$\sK_\infty$ is a strengthening\footnote{For instance, the lower bound in~\eqref{eq:f_sector_1'} is somewhere between the concepts of strictly- and strongly- monotone in the terminology of~\cite[p.97]{deimling2010nonlinear}.} of the monotonicity of~$f$ in its second variable, capturing a lower bound for the rate of increase of~$f$. The inequality in~\ref{ls:A5} is the same as that in the strong monotonicity concept of~\cite[p.97]{deimling2010nonlinear}. However, in each of these cases, note the asymmetry in~$y$ and~$z$ in~\eqref{eq:f_sector_1'}, as the lower bound is only required to hold for all~$z$ in a compact set~$\Gamma$.

%
%
\begin{remark}
    We remark that there are global versions of~\ref{ls:A2}--\ref{ls:A5}, which are {\em a forteriori} sufficient for these respective assumptions. Indeed, suppose that~$\alpha : \mR_+ \to \mR_+$ is such that
\begin{equation}\label{eq:f_sector_1}
\| z_1 -z_2 \| \alpha(\| z_1 -z_2 \|) \leq \langle z_1 -z_2, f(t,z_1) - f(t,z_2) \rangle \quad \forall \: z_1, z_2 \in \mR^m, \; \forall \: t \geq 0\,.
\end{equation}
Hypotheses~\ref{ls:A2}, \ref{ls:A3} or \ref{ls:A5} hold if~$\alpha \in \sP$,~$\alpha \in \sK_\infty$, or~$\alpha \in \sK_\infty$ with~$\alpha(s) \geq \nu s$ for some~$\nu > 0$, respectively.

Similarly, a sufficient condition for~\ref{ls:A4} is that there exist~$\mu, c \geq 0$ such that
\begin{align}
 \| f(t,z_1) - f(t,z_2) \| \leq c\langle z_1 -z_2, &\,  f(t,z_1) - f(t,z_2) \rangle \notag \\
 & \quad \forall \: z_1, z_2 \in \mR^m \quad \text{such that} \quad \| z_1 -z_2 \| > \mu, \; \forall \: t \geq 0\,. \label{eq:f_sector_2}
\end{align}
Consequently, whilst in principle a global version of Theorem~\ref{thm:incremental_passive_SG} may be derived, which imposes~\eqref{eq:f_sector_1}, \eqref{eq:f_sector_2} and the global version of hypothesis~\ref{ls:A1}, namely that there exists~$\theta \in \sK_\infty$ such that
\[ 
\sup_{t \geq 0} \| f(t,z_1) - f(t,z_2) \| \leq \theta(\| z_1 - z_2\|)  \quad \forall \: z_1, z_2 \in \mR^m\,,
\]
this final assumption is highly restrictive. Indeed, the right-hand side of the above is bounded on sets of the form~$\{ (z_1,z_2) \in \mR^m \times \mR^m \: : \: \| z_1 - z_2 \| = L\big\}$, for each fixed~$L>0$. However, the left-hand side of the above need not be bounded on such sets, particularly when~$\| z_1 \|$,~$\| z_2\| \to \infty$.  For example, the function~$f$ which appears in the next example does not satisfy the above condition. It is for this reason that we have concentrated on semi-global results.
\end{remark}

We present two examples of functions which satisfy hypotheses from~\ref{ls:A1} to~\ref{ls:A4}.
%
%
\begin{example}\label{ex:power}
   Consider the time-invariant nonlinearity~$f_d : \mR_+ \to \mR_+$ defined by
\begin{equation}\label{eq:f_power_d}
    f_d(z) : = a_0 z + a_1 \fpd{z}{d}      \quad \forall \: z \in \mR\,,
\end{equation}
for~$d \in \mN$,~$a_0 \geq 0$ and~$a_1 >0$. Functions of this form have been proposed as models of nonlinear viscous damping, such as in~\cite[equation (6)]{fusco2014robust}, which itself is based on the experimental law proposed in~\cite{morison1950force}.

We shall focus on~$g_d(z) : = \fpd{z}{d}$ first. We claim that~$g_d$ satisfies hypotheses~\ref{ls:A1}--\ref{ls:A4}. The derivations of these properties are all routine, but tedious, and several symmetry arguments may be used to reduce the number of calculations. First, since~$g_d$ is odd, it follows that
\begin{equation}\label{eq:f_symmetry_1}
    \lvert g_d(y+z) - g_d(z)  \rvert = \lvert g_d(-y-z) - g_d(-z)  \rvert \quad \forall \: y, z \in \mR\,,
\end{equation} 
and 
\[ \langle y , g_d(y+z) - g_d(z) \rangle = \langle (-y) , g_d(-y-z) - g_d(-z) \rangle \quad \forall \: y, z \in \mR\,.\]
Fix nonempty compact~$\Gamma \subseteq \mR$, and without loss of generality assume that~$z \in \Gamma$ implies that~$-z \in \Gamma$, otherwise replace~$\Gamma$ by the compact set~$\Gamma \cup (-\Gamma)$. 

To establish~\ref{ls:A1}, for~$y >0$ and~$z \in \Gamma$, we have that
\begin{align*}
    \frac{\lvert g_d(y+z) - g_d(z) \rvert}{\lvert y \rvert} & = \frac{g_d(y+z) - g_d(z)}{y} = g_d'(\xi) \quad \text{some~$\xi \in (z,z+y)$} \\
    & = (d+1) \lvert \xi \vert^{d} \leq (d+1) \max\{ \lvert z \rvert, \lvert y + z\rvert\}^d \leq (d+1)\big(\lvert y \rvert + \lvert z \rvert \big)^d\,,
\end{align*}
where we have invoked the mean value theorem for~$g_d$. From the above inequality, in conjunction with the symmetry property~\eqref{eq:f_symmetry_1}, we conclude that~\ref{ls:A1} holds with~$\theta_\Gamma$ a suitable polynomial of degree~$d+1$.

We next claim~\eqref{eq:f_sector_1} holds with~$\alpha \in \sK_\infty$, so that both hypotheses~\ref{ls:A2} and~\ref{ls:A3} hold {\em a forteriori}. For which purpose,  consider
\begin{align*}
    \langle z_1 -z_2, g_d(z_1) - g_d(z_2) \rangle = (z_1 - z_2) \big(g_d(z_1) - g_d(z_2)\big) =: h(z_1, z_2) \quad \forall \: z_1, z_2 \in \mR\,,
\end{align*}
which simplifies to
\[ h(z_1,z_2) = (z_1 - z_2)(\fpd{z_1}{d} - \fpd{z_2}{d})\,. \]
Observe that~$h$ is symmetric, in that~$h(z_1,z_2) = h(z_2,z_1)$ for all~$(z_1,z_2) \in \mR^2$. Since~$g_d$ is odd, it also follows that~$h(z_1,z_2) = h(-z_1,-z_2)$ for all~$(z_1,z_2) \in \mR^2$. In light of the symmetry of~$h$, for~$z_1, z_2 \geq 0$ we may without loss of generality assume that~$z_1 \geq z_2 \geq 0$. An induction argument shows that
\begin{align}
    h(z_1,z_2) &= (z_1 - z_2)(z_1^{d+1} - z_2^{d+1}) \geq (z_1 - z_2)^{d+2} \label{eq:x_estimate_1}\,,
\end{align} 
with the induction step following from the estimates
\[ (z_1 - z_2)^{m+1} = (z_1 - z_2)  \sum_{k=0}^m z_1^{m-k} z_2^{k} \geq  (z_1 - z_2) (z_1^{m} + z_2^m) \geq (z_1 - z_2)(z_1^m - z_2^m) \quad \forall \: m \in \mN\,.\]
Next, we have 
\[ h(z_1,z_2) = (z_1 - z_2)(z_1^{d+1} - (-1)^d z_2^{d+1}) \quad \forall \: z_2 \leq 0 \leq z_1\,.\]
If~$d$ is odd
\begin{align}
    h(z_1,z_2) &= (z_1 - z_2)(z_1^{d+1} + z_2^{d+1}) \gs  (z_1 - z_2)\max \big\{ z_1^{d+1}, z_2^{d+1}\big\} \notag \\
    & \gs (z_1 - z_2)\max\big\{ (z_1 - z_2)^{d+1}, (z_1 + z_2)^{d+1} \big\} = (z_1 - z_2)^{d+2}\,. \label{eq:x_estimate_2} 
\end{align}
Here $\gs$ means greater than or equal to a general multiplicative constant independent of other terms appearing. Its use is intended to reduce the number of constants introduced. If~$d$ is even, then
\begin{align}
    h(z_1,z_2) &= (z_1 - z_2)(z_1^{d+1} - z_2^{d+1}) = (z_1 - z_2)^2 \sum_{k=0}^d z_1^{d-k} z_2^{k} \notag \\
    & \gs (z_1 - z_2)^2(z_1^d - z_2^d)\,, \label{eq:x_estimate_3}
\end{align}
and now the case of~$d$ even applies. The conjunction of~\eqref{eq:x_estimate_1}--\eqref{eq:x_estimate_3} and the symmetry properties of~$h$ give that~\eqref{eq:f_sector_1} holds for~$g_d$ with~$\alpha_1(z) := c_d z^{d+1}$. 

It remains to verify~\ref{ls:A4}. Let~$y \geq 0$ and~$z \in \mR$. From~\ref{ls:A1} we have that
\[ \lvert g_d(y+z) - g_d(z)   \rvert \leq y \theta(y) \quad \text{--- a degree~$d+2$ polynomial in~$y$}\,.\] 
From~\eqref{eq:f_sector_1} with~$z_1 := y+z$ and~$z_2 := z$ we have
\[ \langle y, g_d(y+z) - g_d(z)\rangle \gs c_d y^{d+2} \quad \text{--- also a degree~$d+2$ polynomial in~$y$}\,. \] 
Therefore, since the degrees coincide, property~\ref{ls:A4} holds for some~$c_\Gamma, \mu_\Gamma >0$, as required. 

When~$a_0 >0$, the function~$f_d$ satisfies~\ref{ls:A1}--\ref{ls:A4} and, additionally, \ref{ls:A5} holds as well. \hfill~$\square$
\end{example}

The next example concerns multivariable nonlinearities, and shows the hypotheses~\ref{ls:A1}--\ref{ls:A5} hold for diagonal nonlinearities (as described below), if they hold for every component function.
%
%
\begin{example}\label{ex:diagonal}
Consider the nonlinearity~$f : \mR_+ \times \mR^m \to \mR^m$ for~$m>1$ which is diagonal in its second component, that is,~$f$ satisfies
\begin{equation}\label{eq:diagonal}
 \big(f(t,z)\big)_i = f_i(t, z_i) \quad \forall \: z \in \mR^m, \; \forall \: t \geq 0, \; \forall \: i \in \{1,2,\dots,m\}\,,
\end{equation} 
for given component functions~$f_i : \mR_+ \times \mR \to \mR$. We claim that if~$f_i$ satisfy~\ref{ls:A1}--\ref{ls:A5} for every~$i \in \{1,2,\dots,m\}$, then so does~$f$. 

For which purpose, let~$\Gamma \subset \mR^m$ be compact and let~$P_i : \mR^m \to \mR$ be defined by~$P_i z = z_i$ for all~$z \in \mR^m$ which is trivially continuous. Therefore,~$P_i \Gamma \subseteq \mR$ is compact. To verify~\ref{ls:A1}, we estimate that for~$y \in \mR^m$ and~$z \in \Gamma$,
\begin{align*}
    \| f(t,y+z) - f(t,z) \| &\leq \| f(t,y+z) - f(t,z) \|_1 = \sum_{i=1}^m \big \lvert f_i(t,y_i+z_i) - f_i(t,z_i) \big\rvert \\
    & \leq \theta_{P_i \Gamma} \big( \lvert y_i \rvert \big) \leq \max_i \theta_{P_i \Gamma}(\| y \|) =: \theta_\Gamma(\|y\|)\,,
\end{align*}
since~\ref{ls:A1} holds for each component function~$f_i$, for some function~$\theta_{P_i \Gamma} \in \sK_\infty$. Here~$\| \cdot \|_1$ denotes the one-norm, and we have used that the maximum of~$m$ functions each in~$\sK_\infty$ belongs to~$\sK_\infty$. 

For~\ref{ls:A2}, by hypothesis there exist~$\alpha_{P_i \Gamma} \in \sP$ such that, for each~$i \in \{1,2,\dots,m\}$
\[ y_i \big( f_i(t,y_i + z_i) - f_i(t,z_i)\big) \geq \lvert y_i \rvert \alpha_{P_i \Gamma}(\lvert y_i \rvert) \quad \forall \: y_i \in \mR, \; \forall \: z_i \in P_i \Gamma, \; \forall \: t \geq 0.\]
Set~$\alpha_\Gamma := \min_i \alpha_{P_i \Gamma} \in \sP$, since the~$\alpha_{P_i \Gamma} \in\sP$. Therefore, for~$y \in \mR$,~$z \in \Gamma$, we estimate that
\begin{align*}
 \big \langle y, f(t,y + z) - f(t,z) \big\rangle & = \sum_{i=1}^m y_i \big( f_i(t,y_i + z_i) - f_i(t,z_i)\big) \geq \sum_{i=1}^m \lvert y_i \rvert \alpha_{P_i \Gamma}(\lvert y_i \rvert)  \\ 
 & \geq \| y\|_\infty \alpha_{\Gamma}( \|y\|_\infty) \geq \frac{1}{\sqrt{m}} \| y\|\alpha_{\Gamma}( \|y\|/\sqrt{m}) \quad \forall \: t \geq 0\,.
\end{align*}
where we have used that~$\|y\|\leq \sqrt{m} \|y \|_\infty$. By absorbing the positive constants into~$\alpha_\Gamma$ and relabelling, we infer that~\ref{ls:A2} holds. The argument for~\ref{ls:A3} is the same as above, and only uses that~$\alpha_\Gamma \in \sK_\infty$ when~$\alpha_{P_i \Gamma} \in \sK_\infty$.

 The derivations of~\ref{ls:A4} and~\ref{ls:A5} are similar, and are hence omitted for brevity. \hfill~$\square$
\end{example}

\section{Consequences of incremental stability --- response to almost periodic forcing terms}\label{sec:ap}

Here we demonstrate how the incremental stability results of Section~\ref{sec:incremental} facilitate corresponding desirable behaviour of the state of the forced Lur'e equation~\eqref{eq:lure} with respect to almost periodic forcing signals. The results of this section are inspired by those appearing in~\cite[Section 4]{gilmore2021incremental} and, to an extent~\cite[Sections 3 and 4]{gilmore2020infinite} and which trace their roots back to arguments used in~\cite{angeli2002lyapunov}. The key difference is that here the crucial incremental stability properties are ensured by the passivity-type theorems of Section~\ref{sec:incremental}.

We introduce additional notation and terminology required for the results of this section. Let~$R = \mR_+$ or~$\mR$, and let~$X$ denote a Banach space. The situations of interest presently are when~$X$ equals~$\mR^n$ or~$L^1([0,1],\mR^n)$, although there is no additional difficulty working in the general setting. For~$\tau \in R$ we denote by~$\sigma_{\tau} : L^1_{\rm loc}(R, X) \to L^1_{\rm loc}(R,X)$ the shift (or translation) operator by~$\tau$, that is,~$(\sigma_\tau f)(t) = f(t+ \tau)$ for all~$t \in R$. This is a left-shift when~$\tau \geq 0$. With a slight abuse of notation to avoid a proliferation of symbols, we use the same symbol for all such translation operators, that is, independently of the codomain~$X$ in~$L^1_{\rm loc}(R,X)$.

Here we shall consider~\eqref{eq:lure} under the standing assumption that
\[ f(t,z) = f(z) \quad \text{that is,~$f$ is independent of its first variable.} \]
When the forcing terms~$v$ in~\eqref{eq:lure} are defined on all of~$\mR$, we will also make use of the bilateral
behaviour~$\sB\sB$ of~\eqref{eq:lure} defined by
\[
  \sB\sB:=\big\{(v,x)\in L^1_{\rm loc}(\mR,\mR^n)\times
    W^{1,1}_{\rm loc}(\mR,\mR^n)\; : \; \text{$(v,x)$ satisfies~\eqref{eq:lure}
a.e. on~$\mR$}\big\}\,.
\]
A consequence of the assumed time-invariance of~$f$ is that both~$\sB$ and~$\sB\sB$ are shift-invariant, that is,
\begin{equation}\label{eq:shift}
(v,x)\in \sB \: (\sB\sB) \implies (\sigma_{\tau}v,\sigma_{\tau}x)\in \sB \: (\sB\sB) \quad \forall \: \tau \in \mR_+ \: (\mR)\,.
\end{equation}
We define~$BC(R,X)$ and~$BUC(R,X)$ as the spaces of all, respectively, bounded continuous and bounded uniformly continuous functions~$R \to X$. Endowed with the supremum norm,~$BC(R,X)$ and~$BUC(R,X)$ are Banach spaces.

Our present focus is on so-called Stepanov almost periodic functions. For which purpose, we first collect material on almost periodic functions in the sense of Bohr. For further background reading on almost periodic functions, we refer the reader to the texts~\cite{MR0275061,MR0020163,MR2460203}.

A set~$S \subseteq R$ is said to be {\em relatively dense} (in~$R$) if there exists~$l>0$
such that
\[
 [a,a+l]\cap S \neq \emptyset \quad \forall \: a \in R\,.
\]
For~$\eps>0$, we say that~$\tau \in R$ is an~$\eps${\em -period} of
$v\in C(R,X)$ if    
\[
\|v(t)-v(t+\tau)\| \leq \eps \quad \forall \: t \in R\,.
\]
We denote by~$P(v,\eps)\subseteq R$ the set of~$\eps$-periods of~$v$ and we say that~$v\in C(R,X)$ is {\it almost periodic} (in the sense of Bohr) if~$P(v,\eps)$ is relatively dense in~$R$ for every~$\eps>0$. We denote the set of almost periodic functions~$v\in C(R,X)$ by~$AP(R,X)$, and mention that~$AP(R,X)$ is a closed subspace of~$BUC(R,X)$. It is clear that any continuous periodic  function is almost periodic.

The readily established equality
\begin{equation}\label{eq:ap_inf_tail}
\sup_{t\geq\tau}\|v(t)\| =\|v\|_{L^\infty(R)} \quad \forall \: \tau \in R, \; \forall \: v \in AP(R,X)\,,
\end{equation}
shows that functions in~$AP(R,X)$ are completely determined by their ``infinite right tails''.

Let~$C_0(\mR_+,X)$ denote the subspace of~$C(\mR_+,X)$ of functions which converge to zero as~$t \to \infty$. We define 
\[
 AAP(\mR_+,X):=AP(\mR_+,X)+C_0(\mR_+,X)\,,
 \]
 as the space of all {\em asymptotically almost  periodic} functions. It is a closed subspace of~$BUC(\mR_+,X)$ and, since by~\eqref{eq:ap_inf_tail} we have that~$AP(\mR_+,X) \cap C_0(\mR_+,X) = \{0\}$, it follows that~$AAP(\mR_+,X) =AP(\mR_+,X) \oplus C_0(\mR_+,X)$ (direct sum). Consequently, for each~$v \in AAP(\mR_+,X)$, there exists a unique~$v^{\rm ap} \in AP(\mR_+,X)$ such that~$v - v^{\rm ap} \in C_0(\mR_+,X)$ (cf. \cite[Lemma 5.1]{gilmore2020infinite}).

The spaces~$AP(\mR_+,X)$ and~$AP(\mR,X)$ are closely related, as we now briefly recall. Indeed, it is shown in~\cite[Section 4]{gilmore2020infinite}\footnote{Itself following an idea in~\cite[Remark on p.~318]{MR0773063}} that map~$AP(\mR_+,X) \to AP(\mR,X)$,~$v \mapsto v_{\rm e}$ given by
\begin{equation}\label{eq:ap_extension}
v_{\rm e}(t):=\lim_{k \to \infty}v(t+\tau_k) \quad \forall \: t \in \mR\,,    
\end{equation}
where~$\tau_k \in P(v,1/k)$ for each~$k \in \mN$ and~$\tau_k\to\infty$ as~$k\to\infty$ has the following properties:
\begin{itemize}
    \item is well defined, that is,~$v_{\rm e} \in AP(\mR,X)$\,;
    \item~$v_{\rm e}$ extends~$v$ to~$\mR$\,;
    \item~$v \mapsto v_{\rm e}$ is an isometric isomorphism~$AP(\mR_+,X) \to AP(\mR,X)$.
\end{itemize}
We now recall the concept of Stepanov almost periodicity which is weaker than that of Bohr.  For which purpose, recall first the space of uniformly locally-integrable functions~$UL^1_{\rm loc}(R,X)$ by
 \[
 UL^1_{\rm loc}(R,X):=\left\{w\in L^p_{\rm loc}(R,X) \: : \: \sup_{a\in R}\int_{a}^{a+1}\|w(t)\| \, \rd t<\infty\right\}\,.
 \]
It is straightforward to show that~$UL^p_{\rm loc}(R,X)$ when equipped with the Stepanov norm
\[
\|w\|_{S^1}:=\sup_{a\in R} \int_{a}^{a+1}\| w(t) \|\, \rd t \,,
\]
 is a Banach space. The choice of~$1$ in the upper limit of the above integral is also unimportant insomuch as replacing~$a+1$ by~$a+b$ for positive~$b$ gives rise to equivalent norms.

Let~$v\in L_{\rm loc}^1(R,X)$, and~$\eps>0$. We say that\footnote{The terminology {\em~$\eps$-period of~$v$ in the sense of Stepanov} is used in~\cite{gilmore2020infinite} and analogously for the following concepts.}~$\tau\in R$ is a {\em Stepanov~$\eps$-period of~$v$},   if
\[ \|(\sigma_\tau-I)v\|_{S^1} = \sup_{a\in R}\int_a^{a+1}\|v(\theta+\tau)-v(\theta)\| \, \rd \theta\leq\eps\,.\]
The set of Stepanov~$\eps$-periods of~$v$ is denoted by~$P_1(v,\eps)$. The function~$v$ is called {\em Stepanov almost periodic} if, for every~$\eps>0$, the set~$P_1(v,\eps)$ is relatively dense in~$R$. 

We let~$S^1(R,X)$ denote the set of all functions in~$L_{\rm loc}^1(R,X)$ which are Stepanov almost periodic. It is clear that~$AP(R,X)\subset S^1(R,X)$ (where the inclusion is strict, since~$S^1(R,X)$ contains discontinuous functions), and~$P(v,\eps)\subset P_1(v,\eps)$ for every~$v\in AP(R,X)$ and every~$\eps>0$. Moreover, it is readily shown that~$S^1(R,X)$ is a closed subspace of~$UL_{\rm loc}^1(R,X)$ with respect to the Stepanov norm~$\|\,\cdot\,\|_{S^1}$. 

Routine modifications to the map~$v \mapsto v_{\rm e}$ given by~\eqref{eq:ap_extension} shows that it extends to an isometric isomorphism of~$S^1(\mR_+,X) \to S^1(\mR,X)$.

The generalized Fourier coefficients of~$v \in AP(\mR,X)$ are defined by
 \[
  \hat{v}(\lambda):=\lim_{T\to\infty}\frac{1}{2T}\int_{-T}^{T}\e^{-i\lambda t}v(t) \, \rd t \quad\forall\:\lambda\in\mR.
 \]
It is well-known that the above limit exists for all~$\lambda\in\mR$ and the frequency spectrum
 \[
 \sigma_{\rm f}(v):=\{\lambda\in\mR: \hat{v}(\lambda)\not= 0\}
 \]
of~$v$ is countable, see for example,~\cite[Section VI.3]{c89}.
 The module~${\rm mod}(v)$ of~$v\in AP(\mR,X)$ is given by
 \[ {\rm mod}(v) := \Big\{ \sum_{\lambda\in\sigma_{\rm f}(v)}m(\lambda)\lambda \: : \: \text{$m:\sigma_{\rm f}(v)\to\mZ$ has finite support}\Big\}\,,\]
which carries the structure of a~$\mZ$-module and is the smallest subgroup of~$\mR$ containing~$\sigma_{\rm f}(v)$.

The generalized Fourier coefficients of Stepanov almost periodic functions may be defined by appealing to Bochner transform as follows. Given a function~$v\in L^1_{\rm loc}(R,X)$, define~$\tilde v:R\to L^1([0,1],X)$, the so-called Bochner transform of~$v$, by
 \[
 \big(\tilde v(t)\big)(s):=v(t+s)\quad\forall\:t\in R,\,\,\forall\:s\in[0,1].
 \]
Then~$\tilde v\in C(R,L^1([0,1],X))$ and, furthermore,~$v\in S^1(R,X)$ if, and only if,~$\tilde v\in AP(R,L^1([0,1],X))$. In particular, we may define the generalized Fourier coefficients of~$v\in S^1(\mR,X)$ as those of~$\tilde v\in AP(\mR,L^1([0,1],X))$.

%
%
The following proposition is the main result of this section, and is inspired by~\cite[Theorem 4.3]{gilmore2021incremental}.

\begin{proposition}\label{prop:periodic}
Let~$ w \in S^1(\mR_+,\mR^m)$. Consider the Lur'e equation~\eqref{eq:lure}, and assume that:
\begin{enumerate}[label = {\rm (\roman*)}, itemsep = 1ex, topsep = 1ex, leftmargin = 1ex, itemindent = 4ex]
    \item \label{ls:p1} the pair~$(C,A)$ is detectable;
    \item \label{ls:p2} there exists a symmetric positive definite~$P \in \mR^{n\times n}$ which satisfies~\eqref{eq:P_lmi_condition};
    \item \label{ls:p3} hypotheses~\ref{ls:A1} and~\ref{ls:A2} hold, and;
    \item \label{ls:p4} there exists~$(w,z) \in \sB$ with bounded~$z$.
\end{enumerate}
Then there exists a unique~$z^{\rm ap} \in AP(\mR_+,\mR^n)$ such that~$(w, z^{\rm ap}) \in \sB$ and, for every~$\eps >0$, there exists~$\delta>0$ such that~$P_1(w,\delta)\subset P(z^{\rm ap},\eps)$. 

Let~$(v,x) \in \sB$. The following further statements hold.
\begin{enumerate}[label = {\rm (\arabic*)}, itemsep = 1ex, topsep = 1ex, leftmargin = 1ex, itemindent = 4ex]
\item \label{ls:ap_1} If~$v\in L^1(\mR_+,\mR^m)$ with~$v - w \in L^1(\mR_+, \mR^m)$, then
  \begin{equation}\label{eq:ap_convergence}
    \lim_{t\to\infty}\big(x(t)-z^{\rm ap}(t)\big)=0\,,  
  \end{equation}
that is,~$x\in AAP(\mR_+,\mR^n)$ with~$x^{\rm ap}=z^{\rm ap}$.
\item \label{ls:ap_2} If~$w$ is periodic with period~$\tau$, then~$z^{\rm ap}$ is~$\tau$-periodic.
\item \label{ls:ap_3}~$(w_{\rm e},z^{\rm ap}_{\rm e})\in\sB\sB$ and there is no other bounded function~$\hat x : \mR \to \mR^n$ such that~$(w_{\rm e}, \hat x) \in \sB\sB$.
\item \label{ls:ap_4}~${\rm mod}(z^{\rm ap}_{\rm e}) \subseteq {\rm mod}(\widetilde{w_{\rm e}}) = {\rm mod}(w_{\rm e})$, the latter inequality holding if~$w \in AP(\mR_+,\mR^m)$. 
\end{enumerate}

\end{proposition}
The conclusions of Proposition~\ref{prop:periodic} are the same as those of~\cite[Theorems 4.3]{gilmore2021incremental}. The novelty of the current result is that the underlying and essential incremental stability properties come from Theorem~\ref{thm:incremental_passive_SG}, which leverages the passivity structure considered presently. Indeed, Lur'e systems with nonlinearities of the form~\eqref{eq:f_power_d} fall outside of the scope of the results of~\cite{gilmore2021incremental}.

By way of hypotheses, items~\ref{ls:p1}--\ref{ls:p3} appear in Theorem~\ref{thm:incremental_passive_SG}, which is the main result used to establish Proposition~\ref{prop:periodic} above. With regards to hypothesis~\ref{ls:p3}, recall the standing assumption on the resulting~$\alpha_\Gamma$ and~$\theta_\Gamma$ in Remark~\ref{rem:technical}.
 Hypothesis~\ref{ls:p4} is imposed to ensure that every trajectory~$(w,z) \in \sB$ with~$w \in S^1(\mR_+,\mR^m)$ has bounded~$z$. By statement~\ref{ls:SG_IISS} of Theorem~\ref{thm:incremental_passive_SG}, hypothesis~\ref{ls:p4} follows from~\ref{ls:p1}--\ref{ls:p3} if~$f(0) =0$, in which case~$(0,0) \in \sB$.

Extensive commentary is given in~\cite{gilmore2020infinite} and~\cite{gilmore2021incremental} on the results~\cite[Theorem 4.5]{gilmore2020infinite} and~\cite[Theorems 4.3, 4.5]{gilmore2021incremental}, respectively, and much of which is applicable here. To reduce repetition, we do not reproduce the commentary. 
%
%
%
\begin{proof}[Proof of Proposition~\ref{prop:periodic}]
    The proof is the same as that of~\cite[Theorem 4.3]{gilmore2021incremental}, replacing references to~\cite[Theorem 3.3]{gilmore2021incremental} to statement~\ref{ls:SG_IISS} of Theorem~\ref{thm:incremental_passive_SG}. These results give rise to the same inequality~\eqref{eq:incremental_IISS_lure}, valid for the same sets of trajectories. We, therefore, omit the details.
\end{proof}

Proposition~\ref{prop:periodic} leveraged statement~\ref{ls:SG_IISS} of Theorem~\ref{thm:incremental_passive_SG} for the underlying incremental stability result. The next proposition instead leverages statement~\ref{ls:SG_ISS} of Theorem~\ref{thm:incremental_passive_SG} and provides additional convergence properties.
%
%
\begin{proposition}\label{prop:periodic'}
Let~$ w \in S^1(\mR_+,\mR^m) \cap L^\infty(\mR_+,\mR^m)$. Consider the Lur'e equation~\eqref{eq:lure}, and assume that:
\begin{enumerate}[label = {\rm (\roman*)}, itemsep = 1ex, topsep = 1ex, leftmargin = 1ex, itemindent = 4ex]
    \item \label{ls:p1'} the pair~$(C,A)$ is detectable;
    \item \label{ls:p2'} there exists a symmetric positive semi-definite~$P \in \mR^{n\times n}$ which satisfies~\eqref{eq:P_lmi_condition};
    \item \label{ls:p3'} hypotheses~\ref{ls:A1}--\ref{ls:A4} hold, and;
    \item \label{ls:p4'} there exists~$(w,z) \in \sB$ with bounded~$z$.
\end{enumerate}
Then the conclusions of Proposition~\ref{prop:periodic} hold. Moreover, the following further statements are valid:
\begin{enumerate}[label = {\rm (\arabic*)}, itemsep = 1ex, topsep = 1ex, leftmargin = 1ex, itemindent = 4ex]
\item \label{ls:ap_1'} If~$(v,x) \in \sB^\infty$ with~$v\in L^\infty(\mR_+,\mR^m)$ and~$\| v - w \|_{L^\infty(t,\infty)} \to 0$ as~$t \to \infty$, then~\eqref{eq:ap_convergence} holds.
\item \label{ls:ap_2'} There exists~$\theta \in \sK$ such that, for~$(v, x^{\rm ap}) \in \sB^\infty$ with~$v \in S^1(\mR_+,\mR^m) \cap L^\infty(\mR_+,\mR^m)$ and~$x^{\rm ap} \in AP(\mR_+, \mR^m)$
\[ \| x^{\rm ap} - z^{\rm ap} \|_{L^\infty} \leq \theta( \| v - w \|_{L^\infty} )\,.\]
\end{enumerate}
\end{proposition}

\begin{proof}
    The proof is the same as that of~\cite[Theorem 4.5]{gilmore2021incremental}, replacing references to~\cite[Theorem 3.7]{gilmore2021incremental} to statement~\ref{ls:SG_ISS} of Theorem~\ref{thm:incremental_passive_SG}. These results give rise to the same inequality~\eqref{eq:incremental_ISS_lure}, valid for the same sets of trajectories. We, therefore, omit the details.
\end{proof}
%
%
\begin{remark}
Here we comment on a further convergence result in the case that hypothesis~\ref{ls:A5} is additionally imposed in Proposition~\ref{prop:periodic'}. Loosely speaking, in this case a version of statement~(1) of~\cite[Theorem 4.5]{gilmore2020infinite} may be derived, which relates to forcing terms which are only asymptotically Stepanov almost periodic, although now in the Stepanov 2-norm, see~\cite[Sections 2 and 4]{gilmore2020infinite}. For brevity, we do not give a formal statement here.
\end{remark}
%
%
\section{Examples}\label{sec:examples}

We illustrate our results through three examples.

%
%
\begin{example}\label{ex:one_mass}
    As possibly the simplest interesting example, consider first the following model for a forced damped mass-spring system with nonlinear damping force, that is,
    \begin{equation}\label{eq:one_mass}
        m \ddot z + kz + f(\dot z) = v\,.
    \end{equation}
    Here~$z$ denotes the displacement of the mass, and the positive constants~$m$ and~$k$ are the mass and (linear) spring constant, respectively. The terms~$f : \mR \to \mR$ and~$v$ are a nonlinearity, so that~$f(\dot z)$ represents a nonlinear damping force, and an external (excitation) force, respectively. We assume that~$f(0) =0$.

Equation~\eqref{eq:one_mass} may be written as a Lur'e equation~\eqref{eq:lure} by setting
    \begin{equation}\label{eq:one_mass_data}
    x : = \bpm{z \\ \dot z}, \quad A : = \bpm{0 & 1 \\ -k/m & 0}, \quad B := \bpm{0 \\ 1/m}, \quad C := \bpm{0 & 1}\,.
    \end{equation}
    %
%
We proceed to verify the hypotheses of our main results in Sections~\ref{sec:incremental} and~\ref{sec:ap}. The matrix
\[ P : = \bpm{k & 0 \\ 0 & m} = P^\top \,,\]
which is evidently positive definite, satisfies
\begin{equation}\label{eq:one_mass_lmi}
    A^\top P + P A = 0 \quad \text{and} \quad PB - C^\top =0\,,
\end{equation} 
so that in particular~$P$ is a positive definite solution of~\eqref{eq:P_lmi_condition}. Note that~$P$ arises naturally as, along solutions~$x$ of~\eqref{eq:lure}, the associated quadratic form
\[ \frac{1}{2}\langle x(t) , P x(t)\rangle = \frac{1}{2}k z^2(t) + \frac{1}{2} m \dot z^2(t) \quad \forall \: t \geq 0\,,\]
equals the energy stored in the mass-spring system at time~$t$. This, combined with the first equality in~\eqref{eq:one_mass_lmi}, captures that there is no dissipation of energy in the dynamics generated by~$A$ --- the damping terms arise through the inclusion of~$B$,~$C$ and~$f$ in the model. The pair~$(C,A)$ is detectable, as
\[ A - BC = \bpm{0 & 1 \\ -k/m & 0} - \bpm{0\\1/m}\bpm{0 &1 } = \bpm{0 & 1 \\ -k/m & -1/m}\,,\]
is Hurwitz. The remaining hypotheses~\ref{ls:A1}--\ref{ls:A5} which appear in Theorem~\ref{thm:incremental_passive_SG} depend on the nonlinear term~$f$. Example~\ref{ex:power} demonstrates that the power law given by~\eqref{eq:f_power_d} satisfies~\ref{ls:A1}--\ref{ls:A4}, and additionally~\ref{ls:A5} if~$a_0$ in~\eqref{eq:f_power_d} is positive. Since~$f(0) =0$, hypothesis~\ref{ls:p4} from Propositions~\ref{prop:periodic} and~\ref{prop:periodic'} is satisfied (see the commentary after Proposition~\ref{prop:periodic}). In particular, the hypotheses of these results are also satisfied when the suitable assumptions from~\ref{ls:A1}--\ref{ls:A5} hold. 

%
%
In summary, we highlight that under essentially a monotonicity type assumption on the nonlinear term, equation~\eqref{eq:one_mass} exhibits ``nice'' stability and convergence properties. Equation~\eqref{eq:one_mass} is related to the Duffing Equation (see, for example~\cite[Secrtion 5.7]{jordan2007nonlinear}), namely,
\[ m \ddot z + d \dot z + k_1z + k_2 z^3= v\,, \]
which is known to admit chaotic behaviour for certain parameter values. We reconcile this to the present example by noting that, crucially,~$k_2$ is not required to be positive in the Duffing Equation, which leads to a Lur'e equation that does not satisfy the hypotheses of the present work. \hfill~$\square$
\end{example}

%
%
\begin{example}\label{ex:two_mass}
We consider the coupled damped mass-spring system depicted in Figure~\ref{fig:two_masses}. The displacement of the~$i$-th mass is denoted~$z_i$, and each mass is acted on by an external force~$v_i$. As in Example~\ref{ex:one_mass}, the positive constants~$m_i$ and~$k_i$ equal the~$i$-th mass and spring constant, respectively.

%
%
\begin{figure}[h!]
    \centering
  \begin{tikzpicture}
\tikzstyle{spring}=[thick,decorate,decoration={zigzag,pre length=0.3cm,post length=0.3cm,segment length=6}]
\tikzstyle{damper}=[thick,decoration={markings,  
  mark connection node=dmp,
  mark=at position 0.5 with 
  {
    \node (dmp) [thick,inner sep=0pt,transform shape,rotate=-90,minimum width=15pt,minimum height=3pt,draw=none] {};
    \draw [thick] ($(dmp.north east)+(2pt,0)$) -- (dmp.south east) -- (dmp.south west) -- ($(dmp.north west)+(2pt,0)$);
    \draw [thick] ($(dmp.north)+(0,-5pt)$) -- ($(dmp.north)+(0,5pt)$);
  }
}, decorate]
\tikzstyle{ground}=[fill,pattern=north east lines,draw=none,minimum width=0.75cm,minimum height=0.3cm]

\node (M) [draw,outer sep=0pt,thick,minimum width=1cm, minimum height=2cm,rounded corners=1.2ex] {$m_1$};
\node (M2) [draw,outer sep=0pt,thick,minimum width=1cm, minimum height=2cm,rounded corners=1.2ex] at (2.5,0) {$m_2$};

\node (ground) [ground,anchor=north,xshift=1cm,yshift=-0.25cm,minimum width=6cm] at (M.south) {};
\draw (ground.north east) -- (ground.north west);

\node (wall) [ground,anchor=south, minimum height=3cm, xshift=-0.43cm] at (ground.west) {}; 
\draw (wall.north east) -- (wall.south east);

\draw [thick] (M.south west) ++ (0.2cm,-0.125cm) circle (0.125cm)  (M.south east) ++ (-0.2cm,-0.125cm) circle (0.125cm);
\draw [thick] (M2.south west) ++ (0.2cm,-0.125cm) circle (0.125cm)  (M2.south east) ++ (-0.2cm,-0.125cm) circle (0.125cm);

\draw [spring] (M2.220) -- ($(M.north east)!(M.220)!(M.south east)$)node[below, pos = 0.5, yshift = -0.2cm]{$k_2$};
\draw [spring] (M.220) -- ($(M.220) - (1.5,0)$)node[below, pos = 0.5, yshift = -0.2cm]{$k_1$};

\draw [damper] (M.170) -- ($(M.170) - (1.5,0)$)node[above, pos = 0.5, yshift = 0.25cm]{$f_1$};
\draw [damper] (M2.170) -- ($(M.north east)!(M2.170)!(M.south east)$)node[above, pos = 0.5, yshift = 0.25cm]{$f_2$};

\draw[thick, -latex] ($(M) - (0,2)$)--($(M) - (-1,2)$)node[below]{$z_1$};
\draw[thick, -latex] ($(M2) - (0,2)$)--($(M2) - (-1,2)$)node[below]{$z_2$};

\draw[thick, -latex] (M.90) |- ($(M.90) + (1,0.5)$)node[above]{$v_1$};
\draw[thick, -latex] (M2.90) |- ($(M2.90) + (1,0.5)$)node[above]{$v_2$};



\end{tikzpicture}
    \caption{Coupled mass-spring system}
    \label{fig:two_masses}
\end{figure}
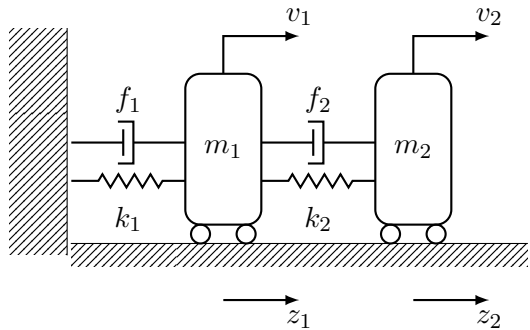

Elementary physical modelling yields the resulting coupled system of forced differential equations
\begin{equation}\label{eq:coupled_msd}
    \left.\begin{aligned}
       m_1 \ddot z_1 + k_1 z_1 + f_1(\dot z_1) - k_2(z_2 - z_1) - f_2(\dot z_2 - \dot z_1) &= v_1 \\
m_2 \ddot z_2 + k_2 (z_2 - z_1)  +  f_2(\dot z_2 - \dot z_1) &= v_2        \,,
\end{aligned}\right\}
\end{equation}
which, introducing
   \begin{equation}\label{eq:msd_lure}
   \begin{aligned}
    x &: = \bpm{z_1 \\ \dot z_1 \\ z_2 \\ \dot z_2}, \quad A : = \bpm{0& 1 &0 &0 \\ -(k_2+k_1)/m_1 &0 &k_2/m_1 &0 \\ 0 &0& 0& 1\\ k_2/m_2 &0 &-k_2/m_2& 0}, \quad B:= \bpm{0 & 0\\ 1/m_1 & -1/m_1 \\ 0 & 0 \\ 0 & 1/m_2}\notag \\
     C &:= \bpm{0 & 1 & 0 & 0 \\ 0 & -1 & 0 & 1}, \quad f(y) := \bpm{f_1(y_1) \\ f_2(y_2)} \quad \forall \: y = \bpm{y_1\\ y_2} \in \mR^2\,,
    \end{aligned}
    \end{equation}
may be written as Lur'e equation~\eqref{eq:lure} with~$n=4$ and~$m=2$.

%
%
We verify the hypotheses of the present work. For which purpose, define
\[ P : = \bpm{k_1+k_2& 0& -k_2& 0\\ 0& m_1&0&0\\ -k_2&0&k_2& 0\\ 0& 0&0&m_2}\,,\]
which is symmetric and positive definite, since~$0 \leq \langle \xi, P \xi \rangle$ for all~$\xi \in \mR^4$ and~$\langle \xi, P \xi \rangle = 0$ necessitates that~$\xi = 0$. As with Example~\ref{ex:one_mass}, the quadratic form induced by~$P$ has the interpretation of stored energy of the coupled damped mass-spring system. A routine calculation shows that~\eqref{eq:one_mass_lmi} holds, so that~$P$ is a positive-definite solution of~\eqref{eq:P_lmi_condition}. To establish detectability, a straightforward calculation gives
\begin{equation}\label{eq:two_mass_detect_1}
     \big\langle \xi, \big((A-BC)^\top P + P (A-BC)\big) \xi \big \rangle = -2 \xi_2^2 - 2(\xi_2 - \xi_4)^2 \quad \forall \: \xi \in \mR^4\,.
\end{equation}
Let~$V(\xi) := \langle \xi, P\xi\rangle$. In light of~\eqref{eq:two_mass_detect_1} and
\[ \langle \nabla V(\xi), (A-BC)\xi \rangle = \big\langle \xi, \big((A-BC)^\top P + P (A-BC)\big) \xi \big \rangle\,, \]
a routine application of Lasalle's Invariance Principle (see, for example~\cite[Theorem 5.12]{MR3288478}) combined with an examination of the linear system of differential equations
\begin{equation}\label{eq:two_mass_detect_2}
\dot \xi = (A-BC)\xi\,,
\end{equation}
proves that every solution of~\eqref{eq:two_mass_detect_2} converges to zero as~$t \to \infty$. In particular,~$A-BC$ is Hurwitz, so that the pair~$(C,A)$ is detectable.

As before, hypotheses~\ref{ls:A1}--\ref{ls:A5} depend on the nonlinearity~$f$. Since~$f$ is diagonal, as in~\eqref{eq:diagonal}, Example~\ref{ex:diagonal} ensures that~$f$ satisfies any of~\ref{ls:A1}--\ref{ls:A5} if both~$f_1$ and~$f_2$ do.

%
%
For a numerical simulation, we fix the following model data and initial conditions
\begin{subequations}\label{eq:msd_coupled_data}
\begin{equation}\label{eq:msd_coupled_linear_data}
\left.\begin{aligned}
    m_1 &= 1.5, \quad m_2 = 0.75, \quad k_1 = 0.5, \quad k_2 = 1.2, \quad f_1(y) = y \lvert y \rvert, \quad f_2(y)  = y \lvert y \rvert^{3/2} \\
     x_0^1 & = \bpm{0.25& 0.25}^\top, \quad x_0^2 = \bpm{ -0.05 & -0.025}^\top\,.
\end{aligned}\right\}
\end{equation}
With these choices hypotheses~\ref{ls:A1}--\ref{ls:A4} are satisfied.
For the forcing terms, first define
\begin{equation}\label{eq:sawtooth}
s(t) : = -1 + \frac{{\rm mod}(t, 2\pi)}{\pi} \quad \text{and} \quad \zeta(t) := t \e^{-3t/2}\;\,.
\end{equation}
Note that the function~$s$ is~$2\pi$-periodic, equal to~$-1$ when~$t=0$, and is continuous from the right. The graph of~$s$ is a (discontinuous) ``sawtooth wave''. Evidently,~$\zeta \in C_0(\mR_+, \mR)$.
Define the forcing terms
\begin{equation}\label{eq:msd_forcing_1}
v_{\rm p}(t) :=  s(0.75t) \bpm{0\\1}, \quad v_{\rm s}(t) := \big(s(0.75t) + s(0.75\sqrt{2}t)\big)\bpm{0\\1},  \quad t \geq 0\,, 
\end{equation}
as well as
\begin{equation}\label{eq:msd_forcing_2}
 v_{\rm ap}(t) := \big(\sin(2\sqrt{2} \pi t) + \sin(2 \pi t)\big)\bpm{0\\1},\quad v_{\rm aap} := v_{\rm s} + \zeta\bpm{0\\1}, \quad t \geq 0\,, 
\end{equation}
\end{subequations}
The forcing terms are structured so that only the second mass is subject to an external force. The functions~$v_{\rm p}$,~$v_{\rm ap}$ and~$v_{\rm s}$ are periodic, almost periodic, and Stepanov almost periodic, respectively. Graphs of the second component of~$v_{\rm p}$ and~$v_{\rm s}$ are contained in Figure~\ref{fig:vp_vs}. The function~$v_{\rm ap}$ is almost periodic as the sum of two periodic functions, but is not periodic as the periods of the summands are not commensurate. Furthermore, note that~$v_{\rm s}$ equals the sum of two periodic functions, so is Stepanov almost periodic. However,~$v_{\rm s}$ is not periodic, for the same reasons as with $v_{\rm ap}$, and is also not almost periodic as not continuous. Finally, the function~$v_{\rm aap}$ is asymptotically almost periodic. 

Numerical simulation results are plotted in Figures~\ref{fig:msd_p}--\ref{fig:msd_s_as}. We let~$z_i(t; z_i(0), v)$ denote the (values of the) solution of~\eqref{eq:coupled_msd} with model data~\eqref{eq:msd_coupled_data} subject to initial condition~$z_i(0)$ and forcing term~$v$. Figures~\ref{fig:p_1} and~\ref{fig:p_2} plot~$z_i(t; z_0^i, v_{\rm p})$ and~$\dot z_i(t; z_0^i, v_{\rm p})$ against~$t$, respectively, subject to either~$z_i(0) = x_0^i$ or~$0$.
Figures~\ref{fig:ap_1} and~\ref{fig:ap_2} plot~$z_i(t; 0, v)$ and~$\dot z_i(t; 0, v)$ against~$t$, respectively, subject to either~$v = v_{\rm ap}$ or~$v_{\rm aap}$. 
Figure~\ref{fig:msd_s_as} plots~$\dot z_i(t; 0, v)$ against~$t$, respectively, subject to either~$z_i(0) = x_0^i$ or~$0$ and~$v = v_{\rm s}$. We plot only $\dot z_i$ here for brevity. In each case, and in accordance with Proposition~\ref{prop:periodic'}, the respective~$z_i$ and~$\dot z_i$ are seen to converge to one another. The limiting behaviour appears periodic (Figure~\ref{fig:msd_p}) or almost periodic (Figures~\ref{fig:msd_ap}--\ref{fig:msd_s_as}).

%
%
We conclude this example by commenting that the current results apply to multiple ($>2$) connected damped mass-spring systems, which generalise the coupled systems presented here. \hfill~$\square$
\end{example}

%
%
\begin{figure}[h!]
    \centering
    \begin{subfigure}{0.32\textwidth}
     \centering
    \includegraphics[width = 0.98 \textwidth]{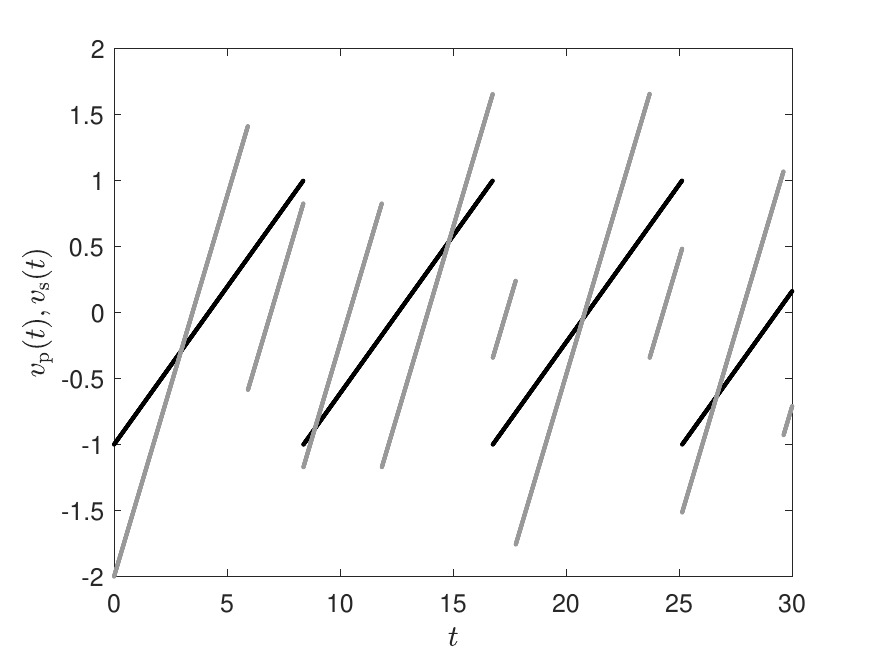}
    \caption{} 
    \label{fig:vp_vs}    
    \end{subfigure}
    \begin{subfigure}{0.32\textwidth}
     \centering
    \includegraphics[width = 0.98 \textwidth]{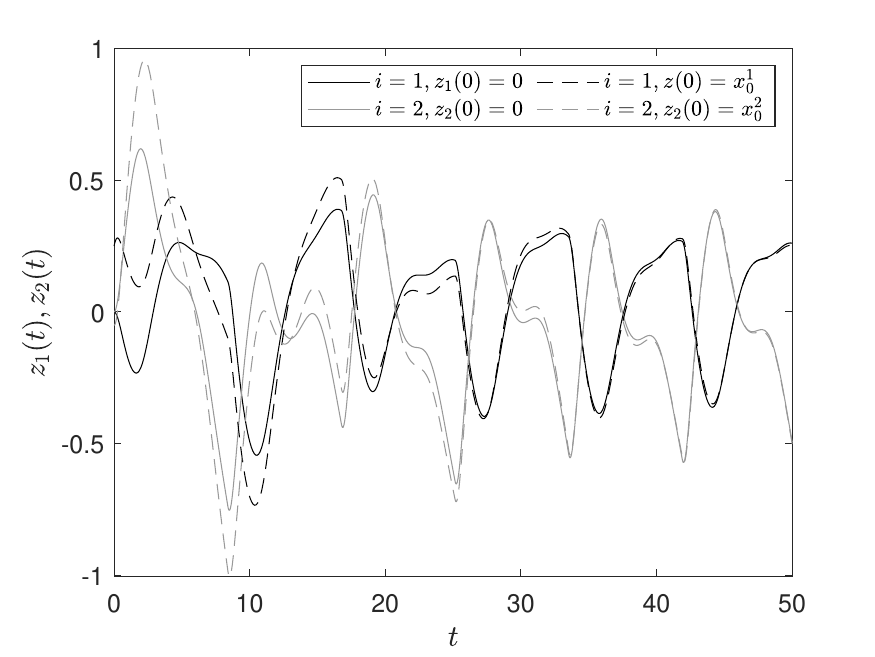}
    \caption{} 
    \label{fig:p_1}    
    \end{subfigure}
    \begin{subfigure}{0.32\textwidth}
     \centering
    \includegraphics[width = 0.98 \textwidth]{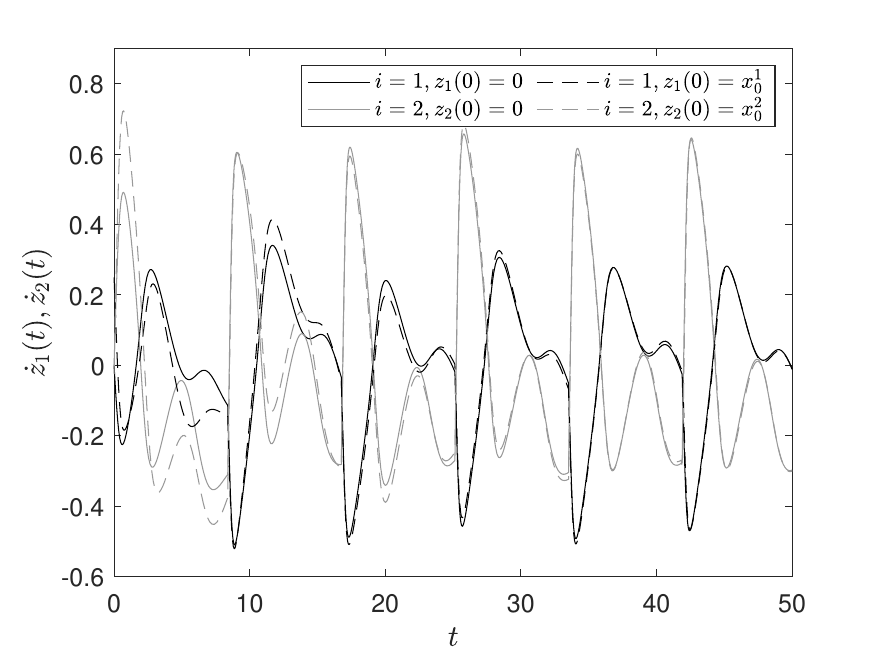}
    \caption{} 
    \label{fig:p_2}    
    \end{subfigure}
    \caption{Numerical results from Example~\ref{ex:two_mass}. (a) Graph of the functions~$v_p$ (black) and~$v_{\rm s}$ (grey). (b) Graphs of~$z_i(t; x_0, v_{\rm p})$ against $t$. (c) Graphs of~$\dot z_i(t; x_0, v_{\rm p})$. See main text for full description.}
    \label{fig:msd_p}
\end{figure}

%
%
\begin{figure}[h!]
    \centering
    \begin{subfigure}{0.975\textwidth}
     \centering
    \includegraphics[width = 0.98 \textwidth]{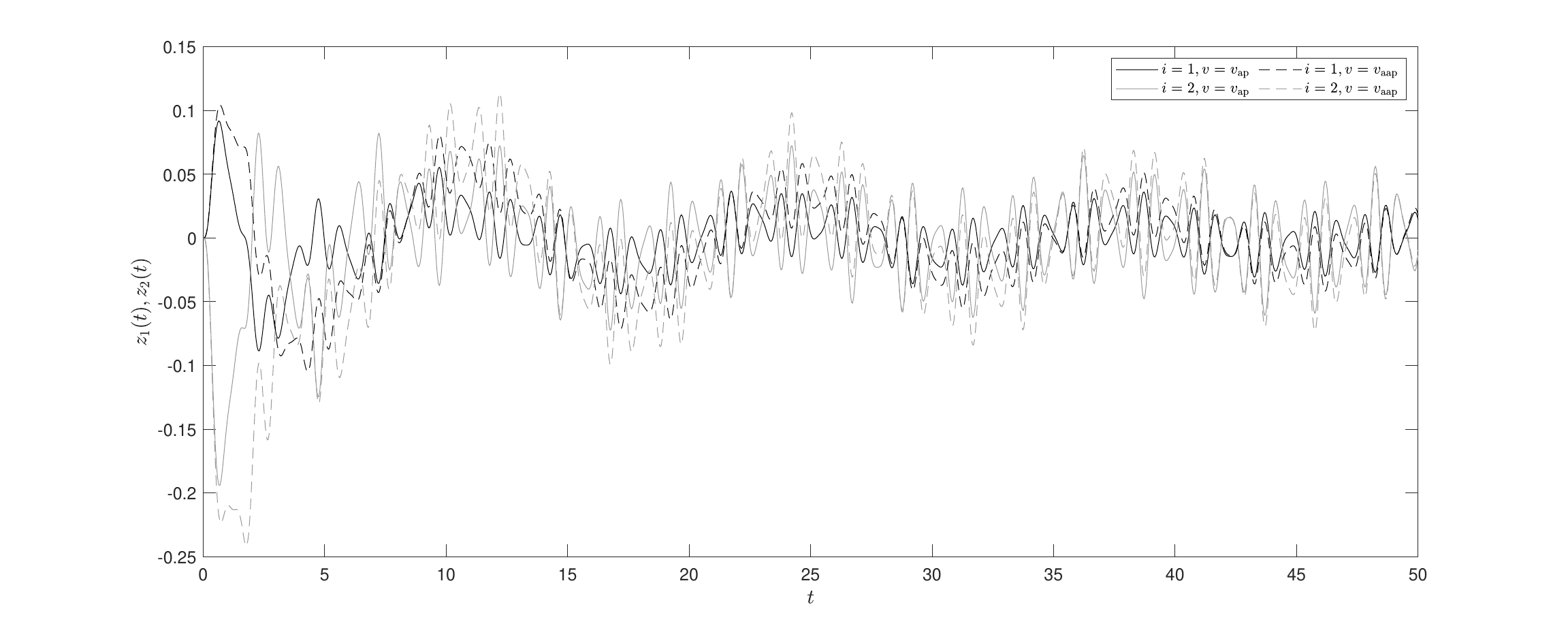}
    \caption{} 
    \label{fig:ap_1}    
    \end{subfigure}\\
    \begin{subfigure}{0.975\textwidth}
     \centering
    \includegraphics[width = 0.98 \textwidth]{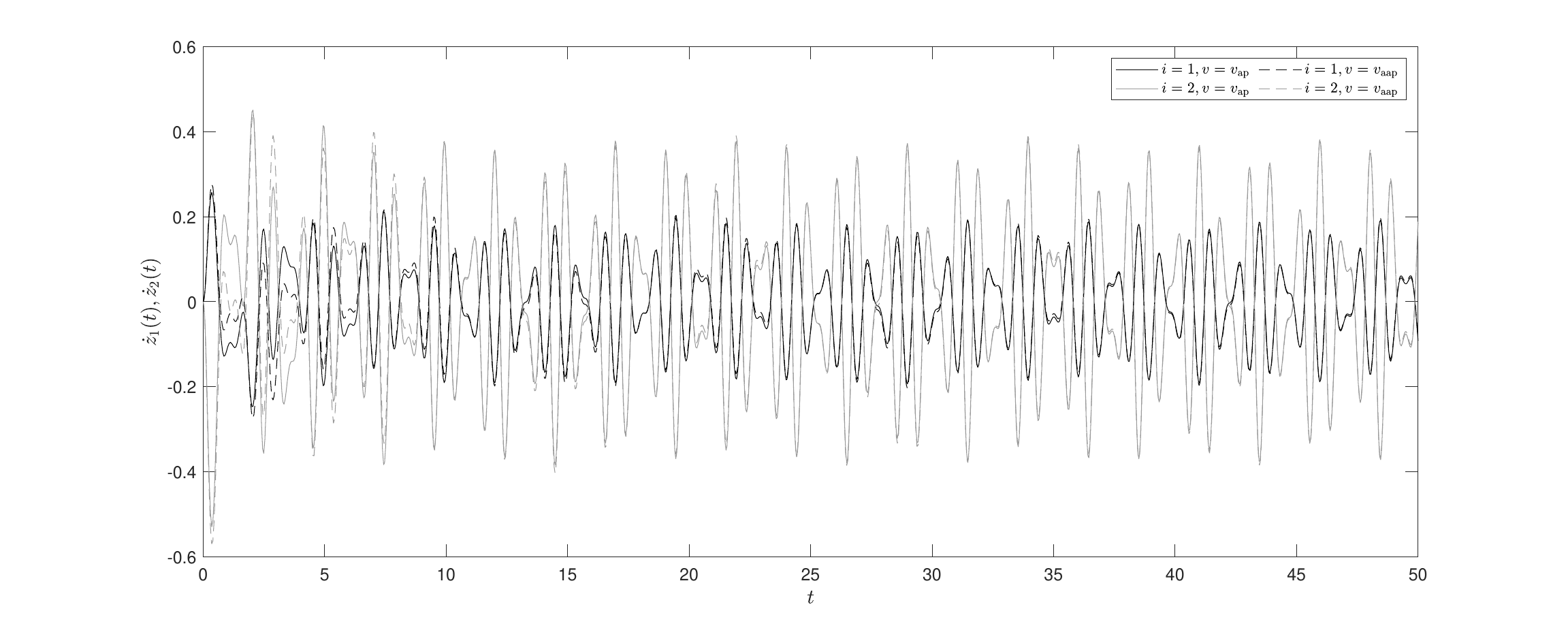}
    \caption{} 
    \label{fig:ap_2}    
    \end{subfigure}
    \caption{(a) Graphs of~$z_i(t; 0, v)$ against $t$. (b) Graphs of~$\dot z_i(t; 0, v)$ against $t$. See legend or main text in Example~\ref{ex:two_mass} for full description.} 
    \label{fig:msd_ap}
\end{figure}

\clearpage
%
%
\begin{figure}[h!]
     \centering
    \includegraphics[width = 0.98 \textwidth]{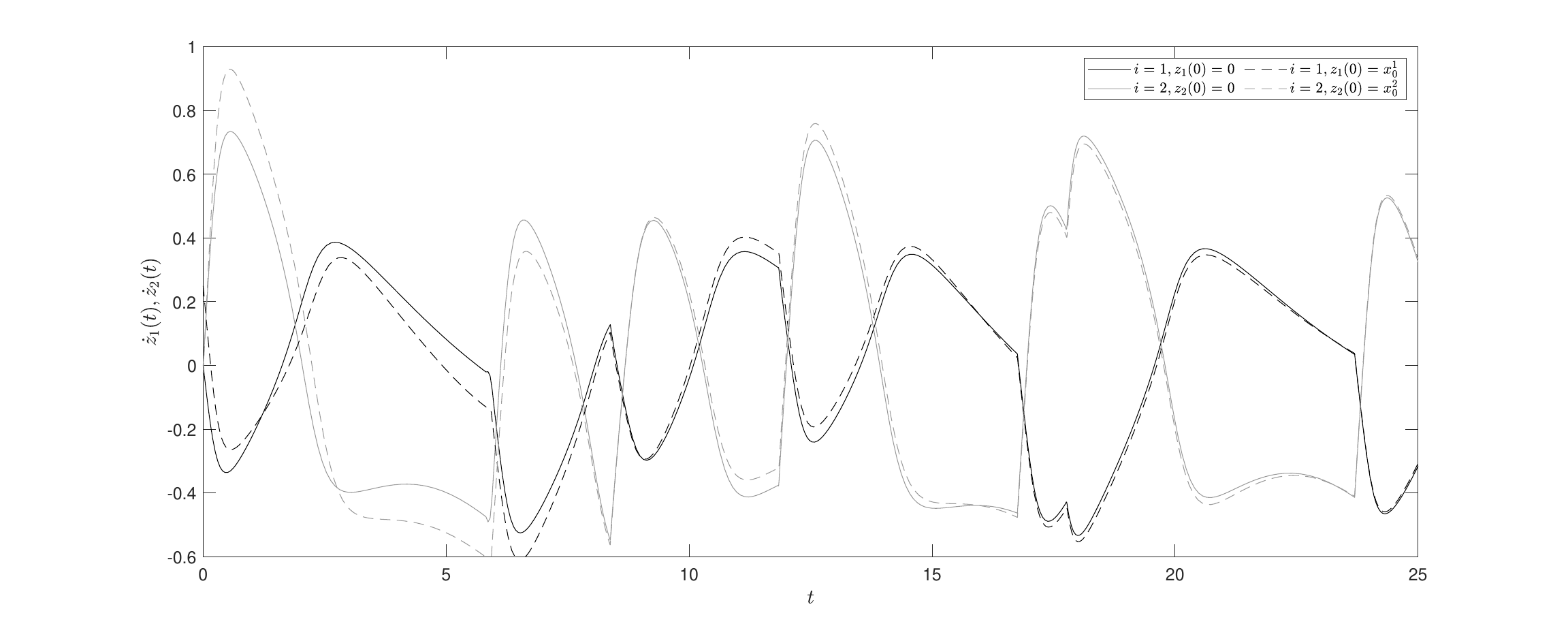}
    \caption{Graphs of~$\dot z_i(t; z_i(0), v_{\rm s})$ against $t$. See legend or main text in Example~\ref{ex:two_mass} for full description.} 
    \label{fig:msd_s_as}
\end{figure}

%
%
\begin{example}\label{ex:pa}
Consider the following model for the vertical displacement~$z$ of a point absorber Wave Energy Converter (WEC) moving in the heave direction only
\begin{subequations}\label{eq:wec_general}
\begin{equation}\label{eq:wec_model}
m \ddot z + f_{\rm rad} + f_{\rm vis} + f_{\rm buoy}  - f_{\rm ex}- f_{\rm PTO}+ f_{\rm moor} = 0\,.
\end{equation} 
The terms in~\eqref{eq:wec_model} are described in Table~\ref{tab:wec_terms}.

\begin{table}[h!]%
\centering
\begin{tabular}{|L{2cm}|L{4.5cm}|} \hline
{\bfseries Term} & {\bfseries Description} \\ \hline 
$m$ & mass of device \\
$f_{\rm rad}$ & radiation force \\
$f_{\rm buoy}$ & buoyancy force \\
$f_{\rm vis}$ & viscous (damping) force \\
$f_{\rm ex}$ & external excitation force \\
$f_{\rm PTO}$ & power take-off force \\
\hline
\end{tabular}
\caption{Description of terms appearing in equation~\eqref{eq:wec_general}}
\label{tab:wec_terms}
\end{table}

Despite included in~\eqref{eq:wec_model} for correctness, we shall in fact for simplicity assume throughout that the mooring force is equal to zero, and so it shall not play a role.  We assume that the buoyancy term is a linear function of displacement, that is,
\begin{equation}\label{eq:linear_buoyancy}
 f_{\rm buoy} = k z\,,
 \end{equation}
for positive constant~$k = g \rho S_w$, the product of acceleration due to gravity~$g$, the water density~$\rho$, and the surface area of the body cut at the mean water level~$S_w$.

In wave-energy conversion applications, the power take-off force is a control variable. 
The excitation force~$f_{\rm ex}$ in~\eqref{eq:wec_model} is a external forcing term, which would be denoted~$v$ in the notation used across the rest of this work. 

The radiation force is usually described mathematically by the Cummins'~\cite{cummins1962impulse} equation
\begin{equation}\label{eq:cummins_equation}
 f_{\rm rad}(t) =  m_\infty \ddot z(t) + \int_0^t h(t-s) \dot z(s) \, \rd s = m_\infty \ddot z + (h_r \ast \dot z)(t)\,, 
\end{equation}
where~$\ast$ denotes convolution, the positive constant~$m_\infty$ is the so-called ``added mass'', and~$h_r$ for some non-parametric kernel (or impulse response). This kernel is typically determined numerically using hydrodynamic software such as the WAMIT~\cite{WAMIT} or FOAMM~\cite{pena2019finite} toolboxes; see~\cite{nuimeprn12466}.

We assume that the viscous drag term~$f_{\rm vis}$ is nonlinear and of the form~\eqref{eq:f_power_d}, as proposed by the experimental law in~\cite{morison1950force}.

To fit the framework of the present paper, we assume a finite-dimensional state-space approximation of order~$n_{\rm r}$ of the convolution term in~\eqref{eq:cummins_equation}, yielding
\begin{equation}
\dot{x}_{\rm r}  =A_{\rm r} x_{\rm r}+B_{\rm r} \dot z,   \quad y_{\rm r} = B_{\rm r}^\top x_{\rm r}\,, \label{eq:radiation_ss}
\end{equation}
and, consequently
\begin{equation}\label{eq:cummins_equation_2}
f_{\rm rad}(t) =  m_\infty \ddot z(t) + y_{\rm r}(t)\,.
\end{equation}
\end{subequations}
The state-space system~\eqref{eq:radiation_ss} is assumed to be stable and passive, here meaning that~$A_{\rm r}$ is Hurwitz and that
\begin{equation}\label{eq:radiation_ss_lmi}
    A_{\rm r}^\top  + A_{\rm r} \leq  0\,,
\end{equation} 
so that~$I$ is a solution of the LMI~\eqref{eq:P_lmi_condition} for the triple~$(A_{\rm r}, B_{\rm r}, B_{\rm r}^\top)$.
Setting~$M := m+m_\infty$ and
\[ x : = \bpm{z \\ \dot z \\ x_{\rm r}}, \quad A : = \bpm{0 & 1 & 0 \\ -k/M & 0 & -B_{\rm r}^\top/M\\ 0 & B_{\rm r} & A_{\rm r}}, \quad B := \bpm{0 \\ 1/M\\0}, \quad C := \bpm{0 & 1 & 0}\,, \]
the PA WEC model~\eqref{eq:wec_general} is of the form
\[ \dot x = Ax - Bf(Cx) + B(v + u)\,,\]
that is, a Lur'e differential equation~\eqref{eq:lure} with~$v$ there replaced by~$u+v$. We verify the hypotheses of our main results. Note that as~\eqref{eq:wec_general} basically comprises a damped mass in feedback with a stable, passive linear control system, there is considerable overlap with the previous examples. Indeed, in light of~\eqref{eq:radiation_ss_lmi}, a straightforward calculation gives that
\[ P : = \bpm{k & 0 & 0 \\ 0 & M & 0 \\ 0 & 0 & I}\,,\]
is a symmetric, positive-definite solution of~\eqref{eq:P_lmi_condition}. The pair~$(C,A)$ is Hurwitz, as can be shown similarly to the argument used in Example~\ref{ex:two_mass}. It has already been established in Example~\ref{ex:power} that~$f$ in~\eqref{eq:f_power_d} satisfies hypotheses~\ref{ls:A1}--\ref{ls:A4}, and~\ref{ls:A5} as well if~$f$ contains a linear term with positive coefficient.

Roughly speaking, in the context of wave-energy, the external signal~$v$ will typically be almost periodic, or at least well-approximated by almost periodic signals. If the control term~$u$ is chosen to also be almost periodic, then Propositions~\ref{prop:periodic} and~\ref{prop:periodic'} apply to~\eqref{eq:wec_general}. These results facilitate the construction of what might be termed ``frequency domain'' methods for such models, and are applied to~\eqref{eq:wec_general} in the context of energy-maximisation optimal control problems in~\cite{guiver2024ccta}. \hfill~$\square$
\end{example}

\section{Summary}\label{sec:summary}

An Input-to-State Stability result has been presented for systems of forced Lur'e differential inclusions which are the feedback connection of two, in some sense, passive components, captured via the linear matrix inequality condition~\eqref{eq:P_lmi_condition}, and certain sign- and sector-conditions on the inclusion~\eqref{eq:F_passive_inclusion}.  Our results for Lur'e inclusions are inspired by corresponding results for forced systems of Lur'e differential equations considered in the seminal work~\cite{arcak2002input}. Our main results for Lur'e inclusions are Theorems~\ref{thm:iss_inclusion} and~\ref{thm:iiss_inclusion}. Interestingly, this latter theorem shows that assumptions which are known to be sufficient for global asymptotic stability are, in fact, sufficient for the {\em a priori} stronger Integral Input-to-State (IISS) property to hold. We comment that it is known that for forced nonlinear differential equations, the IISS property is not equivalent in general to global asymptotic stability of the zero solution of the unforced system 0-GAS (plus forward completeness); see~\cite[Section V]{angeli2000characterization}. 

As with the approach adopted in~\cite{gilmore2021incremental}, Input-to-State Stability results for differential inclusions pave the way for ensuring {\em semi-global incremental} Input-to-State Stability results for systems of forced Lur'e differential equations, where certain uniformity properties with respect to the model data and nonlinear terms are ensured. The passivity assumption on the linear component remains the same as in the Lur'e inclusion case, whilst the condition on the nonlinear term becomes various strengthened versions of the classical monotonicity property. As discussed, in order to treat examples of practical interest, we are required to impose ``semi-global'' assumptions, and derive semi-global results, with Theorem~\ref{thm:incremental_passive_SG} being our main result for Lur'e equations. We reiterate that semi-global results are no less interesting or relevant in all practical situations. Furthermore, the hypotheses of the current work accommodate Lur'e equations which fall outside of the scope of~\cite{gilmore2020infinite,MR4052324,gilmore2021incremental} so that the present work is complementary to these. 

In addition to the independent interest of the stability results mentioned thus far, another motivation for the study is to establish the theoretical underpinnings of what may be termed frequency domain methods for forced Lur'e differential equations. Indeed, the results of Section~\ref{sec:ap} qualify and quantify the state-response to (almost) periodic forcing terms, including those which are Stepanov almost periodic. 
The frequency domain is often preferentially used in control engineering, particularly in the timely area of wave-energy conversion, and the present results may help overcome the nonlinearity barrier to deploying such methods. As mentioned in Example~\ref{ex:pa}, applications of the current results have been initiated in the wave-energy setting in~\cite{guiver2024ccta}.

\vspace{0.5cm}

{\bfseries Acknowledgement} 

The author thanks Prof. Hartmut Logemann (University of Bath) who commented on an early draft of this work.

{{\bfseries Funding Sources}}

Chris Guiver's contribution to this work has been supported by a Personal Research Fellowship from the Royal Society of Edinburgh (RSE), award~\#2168. Chris Guiver expresses gratitude to the RSE for their support.

{\bfseries Disclosure statement. }

The author reports there are no competing interests to declare. No generative AI was used in the production of this work.

{\bfseries Data availability statement}

Data sharing is not applicable to this article as no new data were created or analyzed in this study.

\section*{Appendices}
\appendix

We provide details not given in the main text.

\section{Further properties of correspondence $F$}\label{app:F_properties}

We provide arguments for the properties of the correspondence~$F$ defined in~\eqref{eq:F_passive_inclusion}.

{\bfseries~$F$ is nonempty-valued:} \; Evidently,~$F(0) = \{0 \} \neq \emptyset$. For~$y \neq 0$, we claim that~$w := \theta(\|y\|)y/\|y\| \in F(y)$. Indeed,
\[ \| w \| = \theta(\|y\|), \quad \langle w, y \rangle = \| y \|\theta(\| y \|) \geq \| y \| \alpha(\|y\|)\,,\]
and, if~$\|y \| > \mu$, then
\[ c\langle w, y \rangle = c\| y \|\theta(\| y \|) \geq c \mu \theta(\| y \|) \geq \theta(\| y \|) = \| w \|\,.\]
{\bfseries~$F$ is compact-valued:} \; It is obvious that~$F(y)$ is closed and bounded, hence compact, for all~$y \in \mR^m$.

 {\bfseries~$F$ is convex-valued:} \; Let~$y \in \mR^m$,~$w_1, w_2 \in F(y)$ and~$t \in [0,1]$. We compute that
\[ \| tw_1 + (1-t) w_2 \|\leq t \| w_1\| + (1-t) \| w_2\| \leq \big(t + (1-t)\big) \theta(\|y\|) = \theta(\|y\|)\,.\]
Similarly,
\[ \langle tw_1 + (1-t)w_2 ,y\rangle = t\langle w_1 , y\rangle + (1-t)\langle w_2 ,y\rangle \geq \big(t + (1-t)\big) \| y \| \alpha(\|y \|) = \| y \| \alpha(\|y\|)\,.\]
Finally, if~$\|y \| > \mu$,
\[ \| tw_1 + (1-t)w_2  \| \leq t \| w_1\| + (1-t) \| w_2\| \leq t \langle w_2,y\rangle + (1-t) \langle w_2,y\rangle = \langle tw_1 + (1-t)w_2,y\rangle\,.\]
We conclude that~$tw_1 + (1-t)w_2 \in F(y)$, so that~$F(y)$ is convex.

 {\bfseries~$F$ is upper-hemicontinuous:} \; Since~$F$ takes compact values, we use the characterisation of upper hemicontinuity in~\cite[Theorem 17.20, p.\ 565]{MR2378491}. In particular, let~$y \in \mR^m$ and suppose that~$(y_k, w_k) \in \mR^m \times F(y_k)$ satisfies~$y_k \to y$ as~$k \to \infty$. We need to show that~$w_k$ has a limit point in~$F(y)$. For which purpose, since 
 \[ \| w_k \| \leq \theta(\|y_k\|) \leq \theta(\|y\|) +1 \quad \forall \: k \in \mN, \; \text{sufficiently large}\]
 it follows that~$(w_k)_k$ is bounded and, hence, has a convergent subsequence with limit~$w \in \mR^m$, which we pass to without relabelling. We verify that~$w \in F(y)$. Since~$w_k \in F(y_k)$ for every~$k \in \mN$, and by continuity of the norm and~$\theta$, we have
 \[ \| w \| = \lim_{k \to \infty} \| w_k \| \leq \lim_{k \to \infty} \theta(\|y_k \|)  = \theta(\lim_{k \to \infty} \| y_k \|) = \theta(\|y\|)\,.\]
 Next, and as~$\alpha$ is continuous,
 \[ \langle w, y \rangle = \lim_{k \to \infty} \langle w_k, y_k \rangle \geq \lim_{k \to \infty} \| y_k \| \alpha(\| y_k\|) = \|y \| \alpha(\|y\|)\,.\]
Finally, if~$\| y \| \leq \mu$, then the above two inequalities show that~$w \in F(y)$. If instead~$\| y \| > \mu$, then~$\| y_k\| > \mu$ for all sufficiently large~$k \in \mN$. Therefore,
\[ \| w\| = \lim_{k \to \infty} \|w_k \| \leq \lim_{k \to \infty} c \langle w_k, y_k \rangle = c\langle w, y \rangle\,, \]
again showing that~$w \in F(y)$. We conclude that~$F$ is upper-hemicontinuous at~$y$. Since~$y \in \mR^m$ was arbitrary,~$F$ is upper-hemicontinuous on~$\mR^m$. The proof is complete.


\section{Proof of Lemma~\ref{lem:inclusion_F0_prep}}


The arguments that~$F_0$ is non-empty-, compact- and convex-valued, as well as upper-hemicontinuous, are the same as those for~$F$ given in Appendix~\ref{app:F_properties}, and so we do not repeat them here. It remains to show that~$F_0$ is locally Lipschitz, recall meaning that for all bounded, non-empty~$Y \subset \mR^m$, there exists~$L >0$ such that
\begin{equation}\label{eq:set_valued_local_lipschitz}
F_0(y_1)\subset F_0(y_2)+L\|y_1-y_2\| B(0,1) \quad \forall \: y_1,y_2\in Y\,.    
\end{equation}
Since~$F_0$ takes compact values, being locally Lipschitz is equivalent to~$F_0$ being locally Lipschitz as a map~$\mR^m \to (\sC, d_H)$, where~$(\sC, d_H)$ is the metric space of compact subsets of~$\mR^m$ equipped with the Hausdorff metric. For which purpose, let bounded, non-empty~$Y \subset \mR^m$ be given, and let~$y_1, y_2 \in Y$. Since~$F_0(0) = \{0\}$, the claim is trivial if either~$y_1 = 0$ or~$y_2 =0$. The following arguments are symmetric in~$y_1$ and~$y_2$. Thus, and without loss of generality, we may assume that~$\| y_2 \| \geq \|y_1\| >0$. 

For brevity, we write 
\[ \alpha_i := \alpha(\|y_i\|), \quad \theta_i := \theta(\|y_i\|) \quad \text{and} \quad r_i := r(\|y_i\|) =  \sqrt{\theta^2(\|y_i\|) - \alpha^2(\|y_i\|)}\,.\]
Our assumptions on~$y_1$ and~$y_2$ ensure that~$0 < \alpha_1 \leq \alpha_2$,~$0 < \theta_1 \leq \theta_2$ and~$r_1 \leq r_2$. Moreover, recall that~$r$ is non-decreasing and locally Lipschitz by hypothesis~\eqref{eq:technical}.

By orthogonal decomposition, we have that~$\mR^m = \langle y_i \rangle \oplus \langle y_i \rangle^\perp$ for~$i = 1,2$. Here $\langle \cdot \rangle$ denotes linear span. Therefore, any~$w_i \in F_0(y_i)$ may be expressed as
\begin{subequations}\label{eq:wi_F0}
\begin{equation}
    w_i = a_i \frac{y_i}{\|y_i\|} + b_i \frac{y_i^\perp}{\|y_i^\perp \|} \quad i = 1, 2\,,
\end{equation} 
for some non-zero~$y_i^\perp \in \langle y_i \rangle^\perp$ with~$a_i, b_i \in \mR$. There is no loss of generality in assuming that~$y_i^\perp \neq 0$, as~$b_i =0~$ is possible. These constants must satisfy
\begin{equation}\label{eq:wi_F0_i}
 a_i \geq \alpha_i \quad \text{so that} \quad \langle y_i, w_i \rangle \geq \|y_i \| \alpha_i\,,
 \end{equation}
and
\begin{equation}\label{eq:wi_F0_ii}
\lvert a_i \rvert^2 + \lvert b_i \rvert^2 \leq \theta^2_i\,.
\end{equation}
\end{subequations}
Note that when~\eqref{eq:wi_F0_ii} holds with equality, then~\eqref{eq:wi_F0_i} yields that~$\lvert b_i \rvert^2 \leq r_i^2$, with equality if~$a_i = \alpha_i$. Conversely, observe that any~$w_i \in \mR^m$ satisfying~\eqref{eq:wi_F0} belongs to~$F_0(y_i)$.

Let~$w_1$ be as in~\eqref{eq:wi_F0}. We consider two exhaustive cases.

{\sc Case 1.~$y_2$ is parallel to~$y_1$:} In this case~$y_2 = c y_1$ for some~$c \in \mR$. Let
\[ \Big\{ \frac{y_1}{\|y_1\|}, x_2, \dots , x_m \Big\} \quad \text{and} \quad \Big\{\frac{y_1^\perp}{\|y_1^\perp\|}, z_2, \dots, z_m \Big\}\,, \]
denote two bases for~$\mR^m$ of unit vectors, and define the linear operator~$R : \mR^m \to \mR^m$ by
\[ R \Big(\frac{y_1}{\|y_1\|} \Big) := \frac{y_1^\perp}{\|y_1^\perp\|}, \quad  R x_k := z_k \quad k \in \{2,\dots,m\}\,,\]
and extend to all of~$\mR^m$ by linearity. Note that~$\langle y_2, Ry_2\rangle = \langle c y_1, c R y_1 \rangle = c^2 \langle  y_1 , y_1^\perp \rangle (\|y_1\| /\| y_1^\perp\|) =0$.

{\sc Case 2.~$y_2$ is not parallel to~$y_1$:} If~$m=2$, then~$R y_1 = y_1^\perp$ for some rotation by~$\pm \pi/2$ (and possible scaling) operator~$R$. In this case~$y_2^\perp := R y_2$ is orthogonal to~$y_2$, and~$y_1^\perp$ and~$y_2^\perp$ are linearly independent, since~$y_1$ and~$y_2$ are. 

For~$m >2$, let~$y_2^\perp$ be orthogonal to~$y_2$ and linearly independent of~$y_1^\perp$ (such a choice is always possible). Thus, let
\[ \Big \{\frac{y_1}{\|y_1\|}, \frac{y_2}{\|y_2\|}, x_3, \dots , x_m \Big \} \quad \text{and} \quad \Big \{\frac{y_1^\perp}{\|y_1^\perp\|}, \frac{y_2^\perp}{\|y_2^\perp\|}, z_3, \dots , z_k \Big \}\,, \]
denote two bases for~$\mR^m$ of unit vectors, and define the linear operator~$R : \mR^m \to \mR^m$ by
\[ R \Big(\frac{y_i}{\|y_i\|} \Big) := \frac{y_i^\perp}{\|y_i^\perp\|}, \; i \in \{ 1,2\}, \quad R x_k := z_k \quad k \in \{3, \dots,  m\}\,,\]
and extend to all of~$\mR^m$ by linearity. In particular,~$Ry_i/\|y_i\| = y_i^\perp /\|y_i^\perp\|$ for~$i \in \{1,2\}$.

In either case, note that~$R$ is a bounded linear operator, independently of~$y_1$,~$y_2$, as 
\[ \Big \| R \Big(\frac{y_i}{\|y_i\|} \Big) \Big\| = 1, \; i \in \{1,2\}, \quad \| R x_k \| = \| z_k\| = 1 \quad k \in \{3, \dots,  m\}\,,  \]
and so~$R$ is bounded on a basis of~$\mR^m$, independently of~$y_1$ and~$y_2$. Define
\[ w_2 : =  a_2 \frac{y_2}{\|y_2\|} + b_2 \frac{y_2^\perp }{\| y_2^\perp \|} = a_2 \frac{y_2}{\|y_2\|} + b_2 \frac{R y_2}{\|y_2 \|}\,,\]
where~$(a_2,b_2)$ satisfy~\eqref{eq:wi_F0_i} and~\eqref{eq:wi_F0_ii}, so that~$w_2 \in F_0(y_2)$, but are otherwise to-be-determined. 

We estimate that
\begin{align*}
    \| w_1 - w_2 \|^2 & = \Big \| \Big( \frac{a_1y_1}{\|y_1\|} - \frac{a_2 y_2}{\|y_2\|} \Big)+ R\Big( \frac{ b_1 y_1}{\| y_1\|} - \frac{b_2 y_2}{\|y_2\|} \Big) \Big \|^2 \\
    & \leq 2 \Big \| \frac{a_1y_1}{\|y_1\|} - \frac{a_2 y_2}{\|y_2\|}  \Big\|^2 + 2 \Big\| R\Big( \frac{ b_1 y_1}{\| y_1\|} - \frac{b_2 y_2}{\|y_2\|} \Big)\Big \|^2 
    \\
    & \leq 4 \lvert a_1 \rvert^2 \Big\| \frac{y_1}{\|y_1\|} - \frac{y_2}{\|y_2\|} \Big\|^2 + 4 \lvert a_1 - a_2 \rvert^2 + 4\|R\|^2\lvert b_1 - b_2 \rvert^2 \notag \\
    & \qquad + 4 \lvert b_1 \rvert^2 \Big \| R \Big(\frac{y_1}{\|y_1\|} - \frac{y_2}{\|y_2\|} \Big)\Big\|^2\,.
\end{align*}
Since~$a_1$ and~$b_1$ are bounded on~$Y$,~$R$ is bounded, and~$z \mapsto z/\|z\|$ is locally Lipschitz on~$\mR^m \backslash\{0\}$, it is clear that the above is bounded by a scalar multiple of~$\| y_1 -y_2\|^2$ once it is established that
\begin{equation}\label{eq:LL_coefficient_condition}
    \lvert a_1 - a_2 \rvert^2 + \lvert b_1 - b_2 \rvert^2 \leq M^2\|y_1 - y_2\|^2\,,
\end{equation}
for some~$M >0$. For which purpose, the diagrams in Figure~\ref{fig:LL_diagrams}, which show regions of the~$(a,b)$-plane, are instructive. The conditions~\eqref{eq:wi_F0_i}--\eqref{eq:wi_F0_ii} translate to the pairs~$(a_i,b_i)$,~$i=1,2$ belonging to the regions~$1$ (green) and~$2$ (blue), respectively. The diagram is symmetric about the~$b=0$ axis, and only the upper half-plane~$b\geq 0$ is shown. Recall the assumption that~$\| y_2 \| \geq \|y_1\|$.

%
%
\begin{figure}[h!]
\centering
\begin{subfigure}[]{0.48\textwidth}
\centering
\begin{tikzpicture}
\draw[->] (-0.5,0) -- (5.0,0) node[right] {$a$} coordinate(x axis);
\draw[->] (0,-0.5) -- (0,5.0) node[above] {$b$} coordinate(y axis);

\def\aone{1};
\def\atwo{3};

\def\rone{2};
\def\rtwo{4};

\def\cone{{sqrt(3)}}; 
\def\ctwo{{sqrt(7)}}; 
\def\angone{deg(1.0472)}; 
\def\angtwo{deg( 0.7227)}; 

\draw [black,thick, dotted,domain=0:(1.2*\rone)] plot (\x, {(sqrt(3))*\x});

\draw [black,thick, dotted,domain=0:4.25] plot (\x, {0.5*\x});
\draw[fill = black] (1.7889,0.5*1.7889)  circle (1.5pt);
\draw[fill = black] (3.5777,0.5*3.5777)  circle (1.5pt);

\draw [thick] (0,\rone) arc(90:0:\rone);
\draw [thick] (0,\rtwo) arc(90:0:\rtwo);

\draw [black,thick,dashed] (\aone,-0.5) -- (\aone,4.5);

\draw [black,thick,dashed] (\atwo,-0.5) -- (\atwo,4.5);

\draw[-,fill opacity=0.25, fill=OliveGreen] (\aone,\cone) arc [start angle=\angone,end angle= 0,radius=\rone] -- (\aone,0) -- cycle;
\draw[-,fill opacity=0.25, fill=blue] (\atwo,\ctwo) arc [start angle=\angtwo,end angle= 0,radius=\rtwo] -- (\atwo,0) -- cycle;

\node (L2) at (-0.25,4) {{\scriptsize~$\theta_2$}};
\node (L1) at (-0.25,2) {{\scriptsize$\theta_1$}};
\node (0) at (-0.25,-0.25) {{\scriptsize$0$}};

\node (La2) at (\atwo,-0.75) {{\scriptsize$\alpha_2$}};
\node (La1) at (\aone,-0.75) {{\scriptsize$\alpha_1$}};

\node (cross) at (\rone,3.25) {{\scriptsize~$a_3$}};

\node (r1) at (1.25, 1) {$1$};
\node (r2) at (3.3, 1) {$2$};

\end{tikzpicture}  
\caption{}
\label{fig:LL_diagrams_a}
\end{subfigure}
%
\begin{subfigure}[]{0.48\textwidth}
\centering
\begin{tikzpicture}
\draw[->] (-0.5,0) -- (5.0,0) node[right] {$a$} coordinate(x axis);
\draw[->] (0,-0.5) -- (0,5.0) node[above] {$b$} coordinate(y axis);
\node (0) at (-0.25,-0.25) {{\scriptsize$0$}};

\def\aone{1};
\def\atwo{1.5};

\def\rone{2};
\def\rtwo{4};

\def\cone{{sqrt(3)}}; 

\def\ctwo{3.7081}
\def\angone{deg(1.0472)}; 
\def\angtwo{deg(1.1864)}; 

\draw [black,thick, dotted,domain=0:(1.2*\rone)] plot (\x, {(sqrt(3))*\x});

\draw[fill = black] (\aone,\cone)  circle (1.5pt);
\draw[fill = black] ({\atwo-0.05},{\ctwo-0.05})   rectangle ++(0.1,0.1);
\draw [thick] (0,\rone) arc(90:0:\rone);
\draw [thick] (0,\rtwo) arc(90:0:\rtwo);

\draw [black,thick,dashed] (\aone,-0.5) -- (\aone,4.5);

\draw [black,thick,dashed] (\atwo,-1) -- (\atwo,4.5);

\draw[-,fill opacity=0.25, fill=OliveGreen] (\aone,\cone) arc [start angle=\angone,end angle= 0,radius=\rone] -- (\aone,0) -- cycle;
\draw[-,fill opacity=0.25, fill=blue] (\atwo,\ctwo) arc [start angle=\angtwo,end angle= 0,radius=\rtwo] -- (\atwo,0) -- cycle;

\node (L2) at (-0.25,4) {{\scriptsize~$\theta_2$}};
\node (L1) at (-0.25,2) {{\scriptsize$\theta_1$}};

\node (La1) at (\aone,-0.75) {{\scriptsize$\alpha_1$}};
\node (La2) at (\atwo,-1.25) {{\scriptsize$\alpha_2$}};

\node (cross) at (1.2*\rone,3.5) {{\scriptsize~$a_3$}};
\node (r1) at (1.25, 1) {$1$};
\node (r2) at (3.3, 1) {$2$};

\draw[pattern=north west lines] (\atwo,{(sqrt(3))*\atwo}) -- (\atwo,3.7081) -- (2,3.4641) -- cycle;

\end{tikzpicture}  
\caption{}
\label{fig:LL_diagrams_b}
\end{subfigure}
    \caption{Section of the~$(a,b)$-plane of coefficients appearing in~\eqref{eq:wi_F0}.}
    \label{fig:LL_diagrams}
\end{figure}
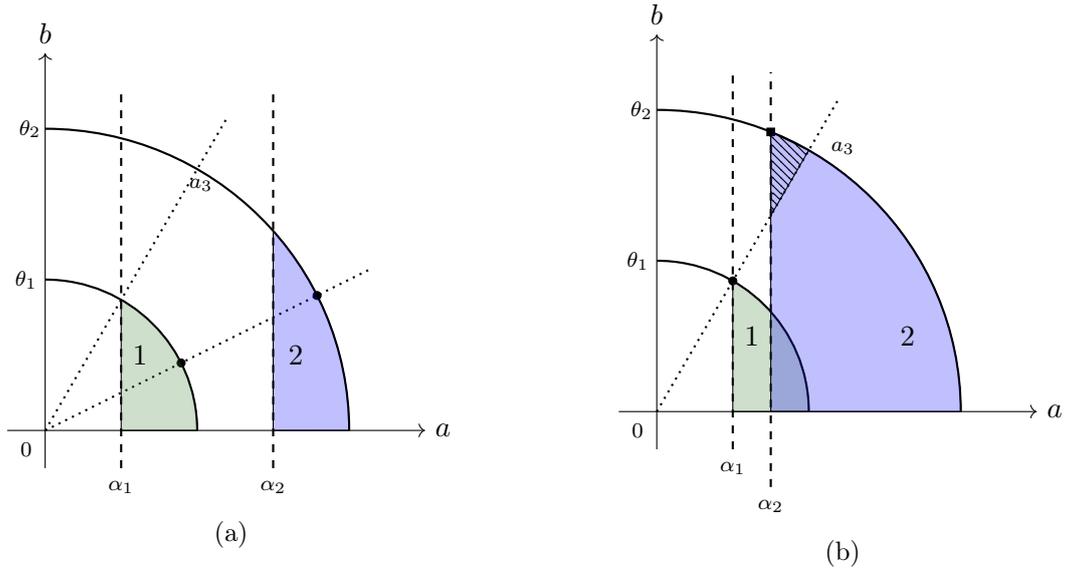
We are considering first where~$w_1 \in F_0(y_1)$ is given, so that~$(a_1,b_1)$ belong to region 1 (including its boundary). Defining~$a_2 := \max\{a_1, \alpha_2\}$ and~$b_2 := b_1$,
we have that~$w_2 \in F_0(y_2)$ and, evidently,
\[ \lvert a_1 - a_2 \rvert^2 + \lvert b_1 - b_2 \rvert^2 \leq \lvert \alpha_2 - \alpha_1 \rvert^2 \leq L_\alpha \|y_1 - y_2 \|^2\,,\]
where~$L_\alpha$ is a local Lipschitz constant for~$\alpha$ on~$Y$. Hence, we have shown that
\begin{equation}\label{eq:LL_dH1}
\sup_{w_1 \in F_0(y_1)} {\rm dist}(w_1, F_0(y_2)) \leq \sup_{w_1 \in F_0(y_1)} \| w_1 - w_2\| \leq  L_1 \|y_1 - y_2\|\,,
\end{equation}
for some~$L_1 >0$.

Conversely, let~$w_2 \in F_0(y_2)$, so that~$w_2$ is of the form~\eqref{eq:wi_F0} for some non-zero~$y_2^\perp \in \langle y_2 \rangle^\perp$, and some~$(a_2,b_2) \in \mR^2$.  We construct bases of~$\mR^m$ and operator~$R$ analogously as above. Define
\[ w_1 := a_1 \frac{y_1}{\|y_1\|} +  b_1 \frac{y_1^\perp }{\| y_1^\perp \|} = a_1 \frac{y_1}{\|y_1\|}  + b_1 \frac{R y_1}{\| y_1\|}\,,\]
(where~$R y_1$ is orthogonal to~$y_1$, by construction of~$R$) and where~$(a_1,b_1)$ satisfy~\eqref{eq:wi_F0_i} and~\eqref{eq:wi_F0_ii}, so that~$w_1 \in F_0(y_1)$, but are otherwise to-be-determined.

It is routine to show that 
\[ a_3 : = \frac{\alpha_1 \theta_2}{\theta_1}\,,\]
is the~$a$-coordinate of the intersection of the straight line through~$(0,0)$ and~$(\alpha_1, r_1)$ with the circle centred at zero of radius~$\theta_2$. This line is depicted in both panels of Figure~\ref{fig:LL_diagrams}. We consider two exhaustive cases.

{\sc Case 1:~$\alpha_2 \geq a_3$.} This scenario is depicted in Figure~\ref{fig:LL_diagrams_a}. In this case, given~$(a_2,b_2)$ in region 2, we choose~$(a_1,b_1)$ such that~$a_1^2 + b_1^2 = \theta_1^2$ to lie on the same straight line through zero, as illustrated by the lower dotted line in Figure~\ref{fig:LL_diagrams_a}. It is then clear that
\[ \lvert a_1 - a_2 \rvert^2 + \lvert b_1 - b_2 \rvert^2 \leq \lvert \theta_2 - \theta_1 \rvert^2 \leq M_1^2 \|y_1 - y_2 \|^2\,,\]
for some~$M_1 >0$, since~$\theta$ is locally Lipschitz. 

{\sc Case 2:~$\alpha_2 < a_3$.} This scenario is depicted in Figure~\ref{fig:LL_diagrams_b}. Now the point~$(a_2,b_2)$ is constrained to lie in the hashed region. The point~$(a_1,b_1) = (\alpha_1, r_1)$ is shown with a dot. It is clear that~$(a_2,b_2) = (\alpha_2, r_2)$, shown with a square, is that in the hatched region which maximises~$\| a - b\|^2$. In this case, we have that
\begin{align*}
    \lvert a_1 - a_2 \rvert^2 + \lvert b_1 - b_2 \rvert^2 & = \lvert \alpha_1 - \alpha_2 \rvert^2 + \lvert r_1 - r_2 \rvert^2 \leq M_2^2 \| y_1 - y_2 \|^2\,,
\end{align*}
for some~$M_1 >0$, since both~$\alpha$ and~$r$ are locally Lipschitz. We conclude that~\eqref{eq:LL_coefficient_condition} holds with~$M := \max\{M_1,M_2\}$. Since~$w_2 \in F_0(y_2)$ was arbitrary,
\begin{equation}\label{eq:LL_dH2}
\sup_{w_2 \in F_0(y_2)} {\rm dist}(w_2, F_0(y_1)) \leq \sup_{w_2 \in F_0(y_2)} \| w_1 - w_2\| \leq  L_2 \|y_1 - y_2\|\,, 
\end{equation}
for some~$L_2 >0$.
The conjunction of~\eqref{eq:LL_dH1} and~\eqref{eq:LL_dH2} shows that
\[ d_H(F_0(y_1),F_0(y_2)) \leq \max\big\{L_1, L_2\big\} \| y_1 - y_2 \| \quad \forall \: y_1, y_2 \in Y\,,\]
as required. The proof is complete. \hfill $\square$

\section{Proof of Lemma~\ref{lem:sg_technical_preparation}}

(i) \; Define~$g : \mR^m \to \mR_+$ by
\[ g(y):= \left\{ \begin{aligned}
    &0 &  y &= 0\\
    &\inf_{\substack{z \in \Gamma}} \frac{\langle y, f(y + z) - f(z) \rangle}{\|y \|} & y &\neq 0\,.
\end{aligned} \right.\]
The function~$g$ trivially satisfies~$g(0)=0$, with~$g(y) > 0$ for all~$y \neq 0$ by hypothesis and, further, ~$g$ is continuous. Continuity of~$g$ is a consequence of~$f$ being locally Lipschitz and the imposed hypotheses. Then define~$\alpha : \mR_+ \to \mR_+$ by
\[ \alpha(s) : = \inf_{\|y \| =s} g(y)\,,\]
which belongs to~$\sP$, by continuity of~$g$. With these definitions we obtain the desired lower bounds
\[\frac{\langle y, f(y + z) - f(z) \rangle}{\|y \|} \geq g(y) \geq \alpha(\|y\|) \quad \forall \: y \in \mR^m, \; \forall \: z \in \Gamma\,,\]
so that
\[ \langle y, f(y + z) - f(z) \rangle \geq \| y\| \alpha(\|y\|) \quad \forall \: y \in \mR^m, \; \forall \: z \in \Gamma\,.\]
(ii) \; We first claim that the current hypotheses lead to~\eqref{eq:A3_sufficient} holding for all~$z \in \mR^m$. For which purpose, let~$z \in \mR^m$ and set~$\zeta := y + z - z_0$. We have that
\begin{align*}
    \frac{\langle y, f(y + z) - f(z) \rangle}{\|y \|} & = \frac{\langle y, f(\zeta + z_0) - f(z_0) \rangle}{\|y \|} + \frac{\langle y, f(z_0) - f(z) \rangle}{\|y \|} \notag \\
& \geq \frac{\langle y, f(\zeta + z_0) - f(z_0) \rangle}{\|y \|}  - \| f(z) - f(z_0)\| \\ 
& \to \infty\,, \notag
\end{align*}
as~$\|y \| \to \infty$ (which occurs if, and only if,~$\| \zeta \| \to \infty$). 

The claim that~$\alpha$ can be chosen to belong to~$\sK_\infty$ can now be shown similarly to as in~\cite[Lemma 2.2]{gilmore2021incremental}. We omit the remaining details. \hfill $\square$

%
%


{\footnotesize
\def\cprime{$'$} \def\cprime{$'$}

}

\end{document}